\documentclass[11pt, twoside]{amsart}
\usepackage{amssymb,amsthm,amsmath}\usepackage{txfonts}  \usepackage{srcltx}
\usepackage[english]{babel}

\newtheorem{teo}{Theorem}[section]
\newtheorem{oss}[teo]{Remark}
\newtheorem{war}[teo]{Warning}
\newtheorem{Prop}[teo]{Proposition}
\newtheorem{lemma}[teo]{Lemma}

\newtheorem{Defi}[teo]{Definition}
\newtheorem{es}[teo]{Example}
\newtheorem{corollario}[teo]{Corollary}
\newtheorem{no}[teo]{Notation}

\newtheorem{claim}[teo]{Claim}

\newcommand{\cc}{_{^{_\HH}}}
\newcommand{\vv}{_{^{_\VV}}}

\newcommand{\res}{\mathop{\hbox{\vrule height 7pt width .5pt depth 0pt
\vrule height .5pt width 6pt depth 0pt\,}}\nolimits}

\def \op{^\bot}

\def \ot{^\top}

\def \DH{h}

\newcommand{\iu}{\imath}
\newcommand{\LL}{\mathop{\hbox{\vrule height .5pt width 6pt depth
0pt \vrule height 7pt width .5pt depth 0pt\,}}\nolimits}
\newcommand{\rr}{_{^{_\mathcal{R}}}}

\newcommand{\ngr}{_{^{_{\mathrm Gr}}}}

\newcommand{\ci}{_{^{_{\HH_i}}}}
\newcommand{\ctr}{_{^{_{\HH_3}}}}
\newcommand{\cd}{_{^{_{\HH_2}}}}

\def \spp{_{^{_{{\pitchfork}}}}}

\newcommand{\ct}{_{^{_{\HH_t}}}}

\def \cin{{\mathbf{C}^{\infty}}}

\def\dim {\mathrm{dim}}
\def\dc {d\cc}

\def\ss{_{^{_{\HS}}}}

\def \cont{{\mathbf{C}}}
\def\sst{_{^{_{\HS_t}}}}

\def\g{\mathit{g}\cc}
\def\f{{f}\cc}
\def\eu {_{^{_{\rm Eu}}}}
\def\dg{\textit{grad}\cc}
\def\qq{\textit{grad}\ss}
\def\ts{_{^{_{\TT S}}}}
\def\tm{_{^{_{\TT M}}}}
\def \nis {\sigma^{n-2}\cc}

\def \per {\sigma^{n-1}\cc}

\def\tut{_{^{_{\TT \UU_t}}}}

\def\tu{_{^{_{\TT \UU}}}}

\def \cn{\textit{w}}

\def\WW{\widetilde{W}}
\def\SC{{C^{\gg}}}

\def\UU{\mathcal{U}}

\def \TB{\tau^{_{\TT\!S}}}

\def \nn{\nu\cc}

\def \nt{\nu^t\cc}

\def \XH{\mathfrak{X}(\HH)}
\def \XX{\mathfrak{X}}

\def \MH{\mathcal{H}\cc}

\def \MS{\mathcal{H}\cc}

\def \MST{(\MS)_t}

\def \P{{\mathcal{P}}}

\def \PH{\P\cc}

\def\ww{\widetilde{w}}

\def \Om{\Omega}
 
\def \Lie{\mathcal{L}}

\def \ee{\mathrm{e}}

\def \tt{\mathrm{t}}

\def \R{\mathbb{R}}
\def \Rn{\mathbb{R}^{\DN}}

\def \div{\mathit{div}}

\def \GG{\mathbb{G}}
\def \gg{\mathfrak{g}}

\def \pert{(\per)_t}

\def\lh{\mathcal{D}\ss}

\def\lll{\mathcal{L}\ss}

\def\tsc{\nabla^{^{_{\TT S}}}}
\def\gs{\nabla^{_{\HS}}}
\def\gc{\nabla^{_{\HH}}}

\def \Tor{{\textsc{T}}}

\def \RC{\textsc{R}}

\def \cn{\textit{w}}
\def\WW{\widetilde{W}}

\def\UU{\mathcal{U}}

\def \nn{\nu_{_{\!\HH}}}
\def \XG{\mathfrak{X}(\GG)}

\def \perh {\sigma^{2n}\cc}
 
\def \Lie{\mathcal{L}}

\def \ee{\mathrm{e}}

\def \Om{\Omega}
\def \Rn{\mathbb{R}^{n}}
\def \R{\mathbb{R}}

\def \cji {c_{j\,i}(x)}
\def \C { C(x):=[\cji]_{j,i},\,\, {j=1,\ldots,m \,,\, i=1,\ldots,n}}

\def \Qdim {Q:=\sum_{i=1}^{k}i\,h_i}

\def \X {X=(X_{1}, \ldots, X_{m_1})}

\def \X0 {X_{1}(0)\!=\!\partial_{x_{1}}, \ldots, X_{m_1}(0)\!=\!\partial_{x_{m_1}}}

\def \HG {\HH\GG}
\def \HS {\HH S}

\def \TG {\mathit{T}\GG}
\def \HH {\mathit{H}}

\def \VV {\mathit{V}}
\def \TT {\mathit{T}}
\def \TS {\mathit{T}S}

\def \grad{\textit{grad}}
\def \C0H{\mathbf{C}_{0}^{\infty}(U,\HG)}
\def \C00{\mathbf{C}_{0}^{\infty}(U)}
\def \C01{\mathbf{C}_{0}^{1}(U)}

\def \L1{d\,\mathcal{L}^n}

\def \H1{\mathcal{H}_{{\bf cc}}^{1}}

\def \exp{\textsl{exp\,}}
\def \log{\textsl{log\,}}

\def \esp{\textsl{exp\,}}

\def \Om{\Omega}
\def \Rn{\mathbb{R}^{n}}
\def \R{\mathbb{R}}

\def \cji {c_{j\,i}(x)}
\def \C { C(x):=[\cji]_{j,i},\,\, {j=1,\ldots,m \,,\, i=1,\ldots,n}}

\def \GG{\mathbb{G}}
\def \gg{\mathfrak{g}}

\def \X {X=(X_{1}, \ldots, X_{m_1})}

\def \X0 {X_{1}(0)\!=\!\partial_{x_{1}}, \ldots, X_{m_1}(0)\!=\!\partial_{x_{m_1}}}

\def \HG {\mathit{H}}

\def \C0H{\mathbf{C}_{0}^{\infty}(\Om,\HG)}
\def \C00{\mathbf{C}_{0}^{\infty}(\Om)}
\def \C01{\mathbf{C}_{0}^{1}(\Om)}

\def \exp{\textsl{exp\,}}

\def \esp{\textsl{exp\,}}

\def\GG{\mathbb{G}}

\begin{document}

\vskip 3cm
\begin{center}
{\bf \LARGE{Stable  H-minimal hypersurfaces}}

\vskip 1cm {\large Francescopaolo Montefalcone\footnote{F. M. has been partially supported by the Fondazione CaRiPaRo Project ``Nonlinear Partial Differential Equations: models, analysis, and control-theoretic problems".}\footnote{The author wishes to thank the anonymous referee for many helpful comments that improved the paper.} }

\end{center}

\markboth{Francescopaolo Montefalcone}{Stable  $\HH$-minimal hypersurfaces}

\section*{Abstract} \small
We  prove some stability results for smooth $\HH$-minimal hypersurfaces immersed in a  sub-Riemannian $k$-step Carnot group $\GG$. The main tools are the formulae for the 1st and 2nd variation of the $\HH$-perimeter measure $\per$. 
\\{\noindent \scriptsize \sc Key words and phrases:}
{\scriptsize{\textsf {Carnot groups; $\HH$-minimal
hypersurfaces; 1st and 2nd variation of the $\HH$-perimeter; stability; characteristic set.}}}\\{\scriptsize\sc{\noindent Mathematics Subject
Classification:}}\,{\scriptsize \,49Q15, 46E35, 22E60.}

\normalsize

\tableofcontents

\section{Introduction}

In\footnote{\bf Warning: \rm In this new version, I have corrected some imprecisions and, more importantly, I have removed the last section of the previous version. The reason for this change have been some non-trivial modifications to the preprint \cite{Montec}. No other substantial  changes have been made.}  Differential Geometry a minimal (hyper)surface of
$\Rn$ (or,  of a Riemannian manifold
$(M^n,\left\langle\cdot,\cdot\right\rangle)$) is a smooth
codimension one submanifold  having zero mean curvature. We recall
that the Riemannian mean curvature $\mathcal H\rr$ of a
hypersurface $S$ is the trace of its 2nd fundamental form $B\rr$,
which is the $\cin$-bilinear form defined as $B\rr(X,
Y):=\left\langle\nabla_XY,\nu\right\rangle$ for every $X,
Y\in\XX(\TS):=\cin(S, \TS)$, where $\nabla$ denotes the
Levi-Civita connection on the ambient space (either  $\Rn$ or $M$)
and $\nu$ is the unit normal vector along $S$. Note that $\mathcal
H\rr=-\div\ts \nu$. Minimal
hypersurfaces turn out to be critical points of the Riemannian
$(n-1)$-dimensional volume $\sigma\rr^{n-1}$ and hence,
studying stability of a minimal hypersurface $S$ means to study
conditions under which $S$ turns out to be a  minimum of the
functional $\sigma\rr^{n-1}$. For this reason
the  2nd variation formula of  $\sigma\rr^{n-1}$ becomes a fundamental tool and, in order to avoid
boundary contributions, we  can use compactly supported  
normal variations of $S$. For an introduction to these topics in
the Euclidean and/or Riemannian setting we refer the reader to the
surveys by Chern \cite{Chern}, Lawson \cite{Lawson} and Osserman
\cite{Osserman}; see also  Simons' paper \cite{Simons}. Finally,
for some results concerning stability of minimal and CMC
hypersurfaces, we would like to mention the papers \cite{DC},
\cite{DCE}, \cite{DC2}, \cite{FCS}.

That of Minimal Surfaces is one of the great chapters of the XX
century Mathematics, above all, because was a rich source of
entirely new ideas and theories such as that of Currents,
introduced by Federer and Fleming \cite{FedererFleming} (see
Federer's fundamental treatise \cite{FE}), that of Sets of Finite Perimeter
created by De Giorgi and its school starting from the pioneering
work of Caccioppoli (see the  book by Giusti
\cite{Giusti} or \cite{AFP}), and that of Varifolds, heavily
inspired by Almgreen and  developed by Allard in \cite{Allard1,
Allard2}. A highly recommended introduction for these topics is,
of course, the book by  Simon \cite{Simon}; see also the survey by
Bombieri \cite{Bombieri} and Morgan's book \cite{Morgan}.

In this paper, we  study some of these problems, in the
sub-Riemannian setting of Carnot groups. We recall that a \it
sub-Riemannian manifold \rm is a smooth $n$-dimensional manifold
$M$, endowed with a  non-integrable distribution $\HH\subset\TT M$
of $h$-planes,  called the {\it horizontal bundle}, on which a
(positive definite) metric $g\cc$ is given. The horizontal bundle
$\HH$ satisfies the \it H\"{o}rmander condition \rm and this
implies the validity of Chow theorem so that, different points can
always be joined by  \it horizontal curves \rm (i.e. curves
that are everywhere tangent to $\HH$). The idea  is simply that,
in connecting two points, we are only allowed to follow horizontal
paths joining them. The \it CC-distance $\dc,$ \rm is then defined
by minimizing the $g\cc$-length of horizontal curves connecting
two given points: this is the distance  used in sub-Riemannian
geometry. As an introduction to these topics, we refer the reader
to Gromov \cite{18}, Montgomery \cite{Montgomery}, Pansu \cite{P1,
P4}, Strichartz \cite{Stric}. In this context, Carnot groups  play
a role similar to
 Euclidean spaces in Riemannian geometry. They serve as a model for the tangent space of a sub-Riemannian manifold and, further,
represent a wide class of examples of these
 geometries. By definition, a \it $k$-step Carnot group  $\GG$ \rm  is a $n$-dimensional, connected, simply connected and nilpotent
Lie group (with respect to a group law $\bullet$ which is polynomial) having a \it $k$-step stratified Lie algebra $\gg\cong\Rn$. \rm This means that $\gg$ splits into a direct sum of vector subspaces satisfying suitable commuting relations. More precisely, we have
$\gg=\HH_1\oplus\HH_2\oplus\cdots\oplus\HH_k,$
$[\HH_1,\HH_i]=\HH_{i+1}$ for every $i=1,...,k-1$ and
$[\HH_1,\HH_i]=0$ for every $i\geq k$, where the brackets $[\cdot,\cdot]$
denote Lie brackets. We assume that $\DH_i=\dim\HH_i
\,(i=1,...,k)$ so that $n=\sum_{i=1}^{k}\DH_i$. The stratification
of $\gg$ can be seen as the algebraic counterpart of the
H\"{o}rmander condition.

 We recall that  Carnot groups are homogeneous groups, in the sense that they admit a family of positive anisotropic dilations modeled on the stratification; see \cite{Stein}.
This richness of geometric structures, makes interesting the study
of Geometric Measure Theory in Carnot groups; see, for instance,
\cite{AR},  \cite{ASCV}, \cite{AKLD}, \cite{BTW},  \cite{CCM},
\cite{GN}, \cite{FSSC3, FSSC4, FSSC5, FSSC6},  \cite{Monte,
Montea, Monteb}, \cite{Mag, Mag2, Mag3},  \cite{MoSC} and
bibliographies therein. We  also cite \cite{Bar}, \cite{vari, vari2}, \cite{vari3}, \cite{gar, gar2}, \cite{G},
 \cite{Pauls},  \cite{HRR}, \cite{R1}, \cite{RR} for  many important results concerning
$\HH$-minimal and/or constant horizontal mean curvature (hyper)surfaces of
the Heisenberg group.  Nevertheless, here we have to remark that
not much is known about the geometry of smooth $\HH$-minimal
hypersurfaces in general groups.

The aim of this paper, which is somehow a continuation of
\cite{Monteb}, is  studying the stability of smooth
$\HH$-minimal hypersurfaces immersed in  $k$-step Carnot groups.
Let us briefly describe our results. 

In Section \ref{prelcar}, we
 fix notation and main definitions concerning Carnot groups.
We   use  a left invariant frame
$\underline{X}:=\{X_1,...,X_n\}$ on $\gg$ adapted to the
stratification and  fix a Riemannian metric
$\left\langle\cdot,\cdot\right\rangle$ making $\underline{X}$
orthonormal (henceforth abbreviated as o.n.).  This frame satisfies some non-trivial commuting
relations encoded by the so-called structural constants $\SC_{i,
j}^{r}:=\left\langle[X_i,X_j],X_r\right\rangle\,\,\forall\,  i, j,
r=1,...,n$. 
Note also that  the (uniquely determined) left invariant
Levi Civita connection $\nabla$ can be expressed in terms of
 structural constants. The projection of  $\nabla$ onto the
horizontal space $\HH$ is denoted by $\gc$ and called horizontal
connection. 

In Section \ref{hyp1} we recall basic facts about
immersed hypersurfaces endowed with the $\HH$-perimeter
measure $\per$. Note that $\per=|\PH\nu|\,\sigma^{n-1}\rr$, where
$\sigma^{n-1}\rr$ is the $(n-1)$-dimensional Riemannian measure,
$\nu$ is the unit (Riemannian) normal along $S$ and $\PH$ is the
projection onto $\HH$.  Let
$\nn=\frac{\PH\nu}{|\PH\nu|}$ be the unit horizontal normal along
$S$ and let $\HS\subset\TS$ be the horizontal tangent space,
which is $(\DH-1)$-dimensional at each non-characteristic point
$p\in S\setminus C_S$, where $C_S:=\{p\in S: |\PH\nu|=0\}$ denotes the
characteristic set. It turns out that $\HH_p=\HH_p S\oplus {\mathrm
span} _\R\{\nn(p)\}$ at each $p\in S\setminus C_S$. This allows us to
define the horizontal 2nd fundamental form by setting
$B\cc(X,Y):=\left\langle\gc_XY,\nn\right\rangle$ for every $X, Y\in \cont^1(S, \HS)$. However, this
object is not symmetric, in general. Thus it can be decomposed in its symmetric and skew-symmetric parts, i.e. $B\cc=S\cc+A\cc$. 

 In Section \ref{HDF} we  
discuss some divergence-type formulae, which  are very
important tools. In particular, these results  enable us to  define the
horizontal tangential operators $\lh$ and $\lll$, which are
analogous, in this SR setting, to tangential divergence $\div\ts$
and Laplacian $\Delta\ts$. An important fact is the validity of the
formula \begin{equation*}-\int_{S}\varphi\lll
\varphi\,\per=\int_{S}|\qq\varphi|^2\,\per\end{equation*}for every
compactly supported function $\varphi\in\cont\ss^2(S\setminus
C_S)\cap W^{1,2}\ss(S;\per)$; see Corollary \ref{GDc} and also Remark \ref{remdiff}. This formula
holds  (a fortiori) whenever $\varphi\in\cont^2(S)$. 

In Section \ref{sec2.1} we  discuss the basic calculations needed to prove
the 1st variation formula for the $\HH$-perimeter $\per$. 
 
Section
\ref{tecsec1} contains some other tools: adapted
frames, connection 1-forms and  lemmata concerning the
horizontal 2nd fundamental form $B\cc$. This material is then used
in Section \ref{varformsec} to discuss and prove  the variational
formulae for $\per$. The presentation here is slightly different
from  \cite{Monteb}. In fact, we have tried to simplify the original
proofs. More importantly, we have corrected a mistake that has caused the loss of some divergence-type terms in the variational formulae proved there; see Remark \ref{mistake}. 
 Furthermore, we have extended the formulae to the characteristic case.

 We say that
a   hypersurface $S$  of class $\cont^2$  is
\it $\HH$-minimal \rm  if its horizontal mean curvature $\MS$ is zero at
each non-characteristic $p\in S\setminus C_S$.  Moreover, it turns out that the \textquotedblleft infinitesimal\textquotedblright 1st variation of $\per$  is given by
$$\Lie_W\per=\left(-\MS\langle W,\nu\rangle +\div\ts \left( W\ot|\PH\nu|-\langle W,\nu\rangle\nn\ot \right)\right)\,\sigma\rr^{n-1},$$where $\Lie_W\per$ is the Lie derivative of $\per$ with respect to the initial velocity $W$  of the  variation and the symbols $W\op,\,W\ot$ denote the normal and tangential components of $W$, respectively. If $\MS$ is $L^1(S; \sigma\rr^{n-1})$, the function $\Lie_W\per$ is integrable on $S$ and the integral of $\Lie_W\per$ on $S$ gives the 1st variation of $\per$. Note  that the 
third term in the previous formula depends on the normal component of $W$.  We stress that this term was omitted in \cite{Monteb}. Using a generalized divergence-type formula, the divergence term  can be integrated on the boundary and, if one use compactly supported variations, it follows that  $\HH$-minimal hypersurfaces are  \textquotedblleft critical points\textquotedblright  of the $\HH$-perimeter functional.

The formula for 2nd
variation of $\per$, which is one of the main results of this paper, will be obtained as a result of a long calculation; see Theorem  \ref{2vg}. This formula will be proved under some further  assumptions concerning integrability of some geometric quantities. For a precise statement, we refer the reader to Section \ref{varformsec}. In the  Heisenberg group $\mathbb H^1$, the 1st variation formula  for characteristic surfaces of class $\cont^2$ was first obtained by Ritor\'{e} and Rosales in \cite{RR}. We also stress that Hurtado, Ritor\'{e} and Rosales \cite{HRR} have proved  a formula for the 2nd variation which is analogous to that stated in Theorem \ref{2vg}.
We also quote  \cite{HP2}, for similar results in a  general sub-Riemannian setting.

Using compactly supported variations together with suitable integrability conditions on the function $\frac{1}{|\PH\nu|}$, the 2nd variation formula takes the following simple form
$$II_S(W,\per)=\int_{S}\left(|\qq\cn|^2 -\cn^2{\mathcal
B}\ts\right)\,\per,$$ where $W$ is the variation vector and
$\cn=\frac{\left\langle W,\nu\right\rangle}{|\PH\nu|}$. Here, we have set
$$\mathcal{B}\ts:=\underbrace{\|S\cc\|\ngr^2 +\|A\cc\|\ngr^2}_{=\|B\cc\|^2\ngr} +\sum_{\alpha\in I\vv} \left\langle
 \left[2 \qq(\varpi_\alpha)-C(\varpi)\left(X_\alpha-\varpi_\alpha\nn\right)\right],C^\alpha
\nn\right\rangle;$$see Corollary \ref{corvar2stat},  Definition \ref{defini} and Definition \ref{movadafr} in Section \ref{hyp1}.
Note that the above  expression involves many geometric quantities such as the
horizontal 2nd fundamental form $B\cc$ (or, its symmetric and skew-symmetric parts $S\cc$ and $A\cc$),  the vertical vector field $\varpi$, defined as
$\varpi:=\frac{\P\vv\nu}{|\PH\nu|}=\sum_{\alpha=\DH+1}^n\varpi_\alpha X_\alpha,$ where $\varpi_\alpha:=\frac{\nu_\alpha}{|\P\cc\nu|}$, and  
the matrices of the
structural constants  $C^\alpha$ and $C(\varpi)=\sum_{\alpha\in I\vv}\varpi_\alpha C^\alpha$.

In Section \ref{PCARITGC} we   state  some further  identities  for constant horizontal mean curvature hypersurfaces. In particular, we  find a family of explicit solutions to the equation $$\lll\varphi+\varphi\,\mathcal B\ts=0.$$
This is a key-point of this paper and, using this fact, the main stability inequality  follows by adapting a standard argument in the Riemannian setting; see, e.g. \cite{FCS}.  In Section \ref{SR1} we  prove the following:
\begin{teo}Let $S\subset\GG$ be a
 $\HH$-minimal hypersurface of class $\cont^3$. If there exists $\alpha\in I\vv=\left\lbrace \DH+1,...,n\right\rbrace$ such that either $\varpi_\alpha>0$ or $\varpi_\alpha<0$ on $S$, then each non-characteristic domain $\Om\subset S$ is stable.
\end{teo}

An immediate application of the previous result is contained in the next:
\begin{corollario}Let  $S\subset \GG$  be a complete $\HH$-minimal hypersurface of class $\cont^3$. If $S$ is a graph with respect to some given vertical direction, then each non-characteristic domain $\Om\subset S$ is stable.
\end{corollario}

Finally, , in order to illustrate our results, an analysis of some (more or less simple) examples is given in
Section \ref{exf}; see, more precisely,  Corollary \ref{c0}, Corollary \ref{c1},   and Corollary \ref{c3}.

 \subsection{Carnot groups}\label{prelcar} A $k$-{\it{step Carnot group}}
$(\GG,\bullet)$ is a connected, simply connected, nilpotent and
stratified Lie group (with respect to a group law $\bullet$) so that its Lie algebra $\gg \cong\Rn$
is a direct sum of slices ${\mathfrak{g}}={\HH}_1\oplus...\oplus
{\HH}_k$  such that $
 [{\HH}_1,{\HH}_{i-1}]={\HH}_{i}\quad(i=2,...,k),\,\,
 {\HH}_{k+1}=\{0\}$.
Let $0$ be the identity of $\GG$ and set
$\DH_i:=\dim{{\HH}_i}$ for $i=1,...,k$ and $\DH_1:=\DH$. Moreover
set $\HH:=\HH_1$ and
${\VV}:={\HH}_2\oplus...\oplus {\HH}_k.$ Note that
$\HH$ and $\VV$ are smooth subbundles
 of $\TG$ called
{\it horizontal}  and {\it vertical}, respectively. The horizontal
space $\HH$ is generated by a frame
$\underline{X\cc}:=\{X_1,...,X_{\DH}\}$ of left-invariant vector
fields, which can be completed to a global  graded,
left-invariant frame $\underline{X}:=\{X_1,...,X_n\}$ for $\gg$.
We  stress that the standard basis $\{\ee_i:i=1,...,n\}$ of $\Rn\cong\TT_0\GG$
can be relabeled to be {\it graded} or {\it adapted to the
stratification}. Note that any left-invariant vector field of
$\underline{X}$ satisfies
${X_i}(x)={L_x}_\ast\ee_i\,(i=1,...,n)$, where ${L_x}_\ast$
denotes the differential of the left-translation at $x\in\GG$. We
fix a Euclidean metric on $\Rn\cong \TT_0\GG$ which makes $\{\ee_i:
i=1,...,n\}$ an o.n.  basis; this metric extends to each tangent space  by left-translations and makes
$\underline{X}$ an
  o.n.  left-invariant frame for $\gg$.  We  denote by
  $g =\left\langle\cdot,\cdot\right\rangle$  this metric and  assume that $(\GG, g)$
  is a Riemannian manifold.

We  use the so-called  {\it exponential coordinates of 1st
kind} so that $\GG$ is identified with its Lie algebra $\gg$,
via the (Lie group) exponential map $\exp:\gg\longrightarrow\GG$.

A {\it sub-Riemannian metric} $g\cc$ is a symmetric positive
bilinear form on the horizontal space $\HH$. The {\it
{CC}-distance} $\dc(x,y)$ between $x, y\in \GG$ is given by
$$\dc(x,y):=\inf \int\sqrt{g\cc(\dot{\gamma},\dot{\gamma})}\,dt,$$
where the infimum is taken over all piecewise-smooth horizontal
paths $\gamma$ joining $x$ to $y$. From now on, we shall choose
$\g:=g_{|\HH}$.

 We recall that Carnot groups are {\it homogeneous groups}, i.e. they admit a
one-parameter group of automorphisms
$\delta_t:\GG\longrightarrow\GG$ for any $t\geq 0$.  By definition, one has
$\delta_t x
:=\exp\left(\sum_{j,i_j}t^j\,x_{i_j}\ee_{i_j}\right),$
for every $x=\exp\left(\sum_{j,i_j}x_{i_j}\ee_{i_j}\right)\in\GG.$ The
{\it homogeneous dimension} of $\GG$ is the integer
 $\Qdim$
coinciding with the {\it Hausdorff dimension} of $(\GG,\dc)$ as a
metric space; see \cite{18}, \cite{Montgomery}.

The {\it structural constants} of $\gg$ associated with
$\underline{X}$ are defined by
$\SC^r_{ij}:=\left\langle [X_i,X_j],
 X_r\right\rangle,\,i,j, r=1,...,n$.
They are skew-symmetric and satisfy Jacobi's identity. The stratification hypothesis on $\gg$ can be restated as follows:
\begin{equation}\label{stcon} X_i\in {\HH}_{l},\, X_j \in
{\HH}_{m}\Longrightarrow [X_i,X_j]\in {\HH}_{l+m}\end{equation}
and so if $i\in
I{\!_{^{_{{\HH}_s}}}}$ and $j\in I{\!_{^{_{{\HH}_r}}}}$, then
\begin{equation}\label{notabilmente}\SC^m_{ij}\neq 0 \Longrightarrow m \in I{\!_{^{_{{\HH}_{s+r}}}}}.\end{equation}
We  set
\begin{itemize}\item$C^\alpha\cc:=[\SC^\alpha_{ij}]_{i,j=1,...,\DH}\in\mathcal{M}_{\DH\times
\DH}(\R) \qquad \forall\,\,\alpha=\DH+1, ..., \DH+\DH_{2}$;\item
$C^\alpha:=[\SC^\alpha_{ij}]_{i, j=1,...,n}\in
\mathcal{M}_{n\times n}(\R) \qquad \forall\,\,\alpha=\DH+1, ...,
n$.\end{itemize}

Now we introduce\footnote{\bf Notation. \rm Let $M$ be a smooth manifold. We shall denote by $\Om^k(M)$ the space of differential $k$-forms on $M$.} the left-invariant co-frame
$\underline{\omega}:=\{\omega_i:i=1,...,n\}$ dual to
$\underline{X}$, i.e. $\omega_i=X_i^\ast\in \Om^1(\GG)$ for every $i=1,...,n$.
In particular, note that the {\it left-invariant 1-forms} $\omega_i$ are
uniquely determined by
$$\omega_i(X_j)=\left\langle X_i,X_j\right\rangle=\delta_i^j\qquad \forall\,\,i,
j=1,...,n,$$ where $\delta_i^j$ denotes the Kronecker delta.

Let $\nabla$ denote the (unique) left-invariant Levi-Civita
connection on $\GG$ associated with the left-invariant metric
$g=\left\langle\cdot, \cdot\right\rangle$. It turns
out that
\begin{eqnarray*}\nabla_{X_i} X_j =
\frac{1}{2}\sum_{r=1}^n( \SC_{ij}^r  - \SC_{jr}^i + \SC_{ri}^j)
X_r\qquad\forall\,\,i, j=1,...,n.\end{eqnarray*}If $X, Y\in \XH:=\cin(\GG, H) $, we  set $\gc_X
Y:=\PH(\nabla_X Y)$. The operation $\gc$ is a partial connection called {\it $\HH$-connection}. We stress that $\gc$ is {\it flat, compatible with the metric $\g$} and {\it
torsion-free} (i.e.
 $\gc_X Y - \gc_Y X-\PH[X,Y]=0\,\, \forall\,\,X, Y\in \XH$); see \cite{Monteb} and references therein.

\begin{no}If $X\in\XX^1(\TG):=\cont^1(\GG,\TG)$, we denote by $\mathcal J\rr X$ the Jacobian matrix of $X$ computed with respect to the left invariant frame $\underline{X}=\{X_1,...,X_n\}$. Moreover, if $X\in\XX^1(\HH)=\cont^1(\GG,\HH)$, we  denote by $\mathcal J\cc X$ the horizontal Jacobian matrix of $X$ computed with respect to the horizontal  left invariant frame $\underline{X\cc}=\{X_1,...,X_\DH\}$; see \cite{Monteb}. 
\end{no}

\begin{oss}[Horizontal curvature tensor $\RC\cc$]\label{rh0}The flatness of $\gc$ implies that horizontal curvature tensor $\RC\cc$ is identically zero, where we recall that $\RC\cc(X,Y)Z:=\gc_Y\gc_XZ-\gc_X\gc_YZ -\gc_{[Y,X]\cc}Z$ for $X, Y, Z\in\XH.$
\end{oss}

Horizontal gradient and horizontal divergence operators are
denoted by $\dg$ and $\div\cc$.

A continuous distance
$\varrho:\GG\times\GG\longrightarrow\R_+\cup\{0\}$ is called {\it
homogeneous} if one has
\begin{eqnarray*}\varrho(x,y)=\varrho(z\bullet x,z\bullet
y)\quad\forall\,x,\,y,\,z\in\GG;\qquad
\varrho(\delta_tx,\delta_ty)=t\varrho(x,y)\quad\forall t>
0.\end{eqnarray*}
 We  recall a fundamental example.
 \begin{es}[Heisenberg groups $\mathbb H^n$]\label{hng}The Lie algebra
$\mathfrak{h}_n\cong\R^{2n+1}$ of the $n$-th Heisenberg group can  be defined by using a left-invariant frame
$\{X_1,Y_1,...,X_i,Y_i,...,X_n,Y_n,T\}$, where
$X_i(p):=\frac{\partial}{\partial x_i} -
\frac{y_i}{2}\frac{\partial}{\partial t}$, $Y_i(p):=\frac{\partial}{\partial y_i} +
\frac{x_i}{2}\frac{\partial}{\partial t}$  for every $i=1,...,n$, and
$T(p):=\frac{\partial}{\partial t}$. Here $p=\exp(x_1,y_1,x_2,y_2,...,x_n,y_n, t)$ denotes the generic point in $\mathbb{H}^n$.
One has $[X_i,Y_i]=T$ {for every} $i=1,...,n$, and all other commutators
vanish. Hence, by definition, $T$ is the center of
$\mathfrak{h}_n$ and
 $\mathfrak{h}_n$ turns out to be  nilpotent and stratified of
step 2, i.e. $\mathfrak{h}_n=\HH\oplus \HH_2$. The
  structural constants of $\mathfrak{h}_n$ are
described by the skew-symmetric $(2n\times 2n)$-matrix
\begin{center}
 $C\cc^{2n+1}:=\left|
\begin{array}{ccccc}
  0 & 1 &    \cdot &  0 &  0 \\
  -1 & 0 &  \cdot &  0 &  0 \\
  \cdot & \cdot & \cdot & \cdot & \cdot \\
  0 & 0 &   \cdot & 0 & 1 \\
  0 & 0 &  \cdot & -1 & 0
\end{array}%
\right| $\end{center} associated
with the skew-symmetric bilinear map
$\Gamma\cc:\HH\times\HH\longrightarrow \R$ given by
$\Gamma\cc(X, Y)=\left\langle[X, Y], T\right\rangle$ for every $X, Y\in\HH$.

\end{es}

\subsection{Hypersurfaces and measures}\label{hyp1} The Riemannian
left-invariant volume form on $\GG$ is defined as
$\sigma^n\rr:=\bigwedge_{i=1}^n\omega_i\in
\Om^n(\GG).$ The measure $\sigma^n\rr$ is the Haar measure of $\GG$ and equals
the push-forward of the usual $n$-dimensional Lebesgue measure
$\mathcal{L}^n$ on $\Rn\cong\TT_0\GG$.

Let $S\subset\GG$ be a hypersurface of class $\mathbf{C}^1$. We say that
$x\in S$ is a {\it characteristic point} whenever
$\dim\,\HH_x = \dim (\HH_x \cap \TT_x S)$.  The
{\it characteristic set} of $S$ is defined as $$ C_S:=\{x\in S :
\dim\,\HH_x = \dim (\HH_x \cap \TT_x S)\}.$$ Note that $x\in S$ is
non-characteristic if, and only if, $\HH$ is transversal to $S$ at
$x$, i.e. $\HH_x\pitchfork\TT_x S$. We here observe that the $(Q-1)$-dimensional CC Hausdorff measure of the
characteristic set $C_S$  vanishes, i.e.
$\mathcal{H}_{CC}^{Q-1}(C_S)=0$; see \cite{Mag2}. In fact, under further regularity assumptions,  it is possible to show much more than that. For instance, if $S$ is of class $\cont^2$,
then the $(n-1)$-dimensional  Riemmanian Hausdorff measure of  $C_S$ is zero; see \cite{bal}.

 Let
$\nu$ denote the unit normal vector along $S$. The
$(n-1)$-dimensional Riemannian measure  is defined as  $\sigma^{n-1}\rr :=(\nu\LL\sigma^{n}\rr)|_{S},$  where  $\LL$ denotes
the ``contraction'' operator on differential
 forms; see Lee's book \cite{Lee},  pp. 334-346.  We recall that $\LL:
\Om^k(\GG)\rightarrow\Om^{k-1}(\GG)$
is defined, for $X\in\XX(\TG)$ and
$\alpha\in\Om^k(\GG)$, by setting
$$(X \LL \alpha)
(Y_1,...,Y_{k-1}):=\alpha(X,Y_1,...,Y_{k-1}).$$

At each non-characteristic point of $S$ the {\it unit $\HH$-normal} along
$S$ is the normalized projection of $\nu$ onto $\HH$, that is $\nn:
=\frac{\PH\nu}{|\PH\nu|}.$  The {\it $\HH$-perimeter form} is the $(n-1)$-differential form $\per$ on $S\setminus C_S$ defined
by $$\per:=(\nn \LL
\sigma^n\rr)|_{S\setminus C_S}.$$ If $C_S\neq{\emptyset}$ we extend
$\per$ to the whole of $S$ by setting $\per\res C_{S}= 0$. Note that $\per= |\PH \nu |\,\sigma^{n-1}\rr$. This follows from the well-known formula $(X\LL\sigma^n\rr)|_S=\langle X,\nu\rangle\,\sigma^{n-1}\rr$ for any $X\in \XX(\TG)$.
 In particular, it turns out that $C_S=\{x\in S:\,
|\PH\nu(x)|=0\}$. Let $\mathcal{S}_{CC}^{Q-1}$ be the $(Q-1)$-dimensional
spherical Hausdorff measure associated with the CC-distance $\dc$.
Then  $\per(S\cap B)=k(\nn)\,\mathcal{S}_{CC}^{Q-1}({S}\cap B)$ for all $B\in \mathcal{B}or(\GG)$, where the density $k(\nn)$, called {\it metric factor},
depends on $\nn$; see \cite{Mag}. The {\it horizontal tangent bundle}
$\HS\subset \TS$ and the {\it horizontal normal bundle} $\nn S$
split the horizontal bundle $\HH$ into an orthogonal direct sum,
i.e. $\HH= \nn \oplus \mathit{H}S$.
We also recall that the stratification of $\gg$ induces a
stratification of
$\TS:=\oplus_{i=1}^k\HH_iS,$ where $\HS:=\HH_1S$; see \cite{18}.

\begin{oss}\label{SERES} We have $\dim \HH_p S=\dim \HH_p-1=\DH-1$ at each point $p\in S\setminus C_S$. Furthermore, note that the  definition of $\HS$ makes sense even if $p\in C_S$, but in such a case $\dim \HH_p S=\dim \HH_p =2n$.
\end{oss}

For the sake of simplicity, in the rest of this section we shall assume, unless otherwise mentioned, that $S\subset\GG$ is a  non-characteristic hypersurface of class $\cont^2$. So let
$\tsc$ be the induced connection on $S$ from $\nabla$. The
tangential connection $\tsc$ induces a partial connection on $\HS$
defined by
$$\gs_XY:=\P\ss\left( \tsc_XY\right)  \quad\,\forall\,\,X,Y\in\XX^1(\HS):=\cont^1(S, \HS).$$
It turns out that  $\gs_XY=\gc_X Y-\left\langle\gc_X Y,\nn\right\rangle\nn$. In the sequel,  $\HS$-gradient and  $\HS$-divergence will be denoted, respectively,  by $\qq$
and $\div\ss$. By definition, the horizontal 2{nd} fundamental form of ${S}$ is the
bilinear map given by
\begin{eqnarray*}{B\cc}(X,Y):=\left\langle\gc_X Y, \nn\right\rangle\end{eqnarray*} for any $X, Y\in\XX^1(\TS)$.
The horizontal mean curvature $\MS$
is the trace of ${{B}\cc}$,
i.e. $\MS:={\rm
Tr}B\cc=-\div\cc\nn$. The torsion $\Tor\ss$ of the
$\HS$-connection $\gs$ is given by
$\Tor\ss(X,Y):=\gs_XY-\gs_YX-\PH[X,Y]$ for any $X,Y\in\XX^1(\HS)$. There is a non-zero torsion because,
in general, $B\cc$ is {\it not symmetric}. Hence $B\cc$ can be
regarded as a sum of two matrices, i.e. $B\cc= S\cc + A\cc,$ where
$S\cc$ is symmetric and $A\cc$ is skew-symmetric. 

\begin{Defi}\label{defini}The  \rm principal horizontal
curvatures \it $\kappa_j$ of $S$, $j\in I\ss$, are  the eigenvalues
of $S\cc$, i.e.  eigenvalues of the symmetric part of
the horizontal 2nd fundamental form $B\cc$. Note that  $\MS=\sum_{j\in I\ss}\kappa_j$. We also define some important geometric objects:
\begin{itemize}\item $\varpi_\alpha:=\frac{\nu_\alpha}{|\PH\nu|}\quad \forall\,\,\alpha=\DH+1,...,n$;
\item
$\varpi\cd:=\frac{\P\cd\nu}{|\PH\nu|}=\sum_{\alpha\in I\cd}\varpi_\alpha
X_\alpha$;\item
$\varpi:=\frac{\P\vv\nu}{|\PH\nu|}=\sum_{\alpha\in I\vv} \varpi_\alpha
X_\alpha$;\item $C\cc(\varpi\cd):=\sum_{\alpha\in I\cd}
\varpi_\alpha\,C^\alpha\cc$;
\item
$ C(\varpi):=\sum_{\alpha\in I\vv} \varpi_\alpha\,C^\alpha$.
\end{itemize}We further denote by $C\ss(\varpi\cd)$ the restriction to the subspace $\HS$ of the linear operator $C\cc(\varpi\cd)$.\end{Defi}
 These objects play an important role in the horizontal geometry of immersed hypersurfaces. For instance, we have to remark that
$A\cc=\frac{1}{2}\,C\ss(\varpi\cd)$; see
 \cite{Monteb}.  Moreover, for any $X,\,Y\in\XX^1(\HS)$ we have
 $
 \Tor\ss(X,Y)=\langle[X,Y],\varpi\rangle\nn
=-\langle C\ss(\varpi\cd) X, Y\rangle$.

\begin{es}[Heisenberg group]\label{hng2} We have $\varpi:=\varpi_T=\frac{\left\langle\nu, T\right\rangle}{|\PH\nu|}$ and $C\cc(\varpi\cd)=\varpi C\cc^{2n+1}$; see Example \ref{hng}.  An elementary computation shows that the skew-symmetric part $A\cc$ of the horizontal 2nd fundamental form $B\cc$ is given by $A\cc=\frac{\varpi}{2}C\ss^{2n+1}$, where  $C\ss^{2n+1}=C\cc^{2n+1}|_{\HS}$. Since $\|C\ss^{2n+1}\|^2\ngr=2(n-1)$, it follows that  $\|B\cc\|^2\ngr=\|S\cc\|^2\ngr+\frac{n-1}{2}\varpi^2$. 
\end{es}

\begin{Defi}\label{movadafr}Let $U\subseteq\GG$ be an open set and let $\UU:=S\cap U$. We  call {\rm adapted frame to $\UU$ on $U$} any o.n. frame
 $\underline{\tau}:=\{\tau_1,...,\tau_n\}$ on $U$  such that
$\tau_1(p):=\nn(p),\,\,\HH_p\UU=\mathrm{span}\{ \tau_2(p),..., \tau_{\DH}(p)\}$ for any $p\in\UU$,
  $\tau_\alpha:=
X_\alpha.$ Furthermore, we  set
$\TB_\alpha:=\tau_\alpha -
\varpi_\alpha\tau_1$ for any $\alpha\in I\vv$. We stress that
$\HS^{\bot}=\mathrm{span}_{\R}\{\TB_\alpha: \alpha\in I\vv\},$
where $\HS^{\bot}$ denotes the orthogonal complement of $\HS$ in
$\TS$, i.e.  $\TS=\HS\oplus\HS^{\bot}.$\end{Defi}

Note that every adapted o.n. frame to a hypersurface is a graded
frame. Clearly, we have that
 $$\underline{\tau}=\{\underbrace{\tau_1}_{=\nn},
 \underbrace{\tau_2,...,\tau_{\DH}}_{\mbox{\tiny{o.n. basis of}}\,\HS},\underbrace{\tau_{\DH+1},...,\tau_n}_{\mbox{\tiny o.n. basis of}\,\VV}\}.$$

\begin{no} Let $n_i:=\sum_{j=1}^{i}\DH_j$. Hereafter, we shall set
 $I\cc=\{1,2,...,\DH\}$,  $I\ci=\{n_{i-1}+1,...,n_i\}$, $I\vv=\{\DH+1,...,n\}$ and
 $I\ss:=\{2,3,...,\DH\}$.\end{no}

 Let
$\underline{\phi}:=\{\phi_1,...,\phi_n\}$ be the dual co-frame of $\underline{\tau}$, i.e.
$\phi_i(\tau_j)=\delta_i^j$ for any $i,
j=1,...,n$, where $\delta_i^j$ denotes the Kroneker delta. The co-frame $\underline{\phi}$ satisfies the
Cartan's structural equations:\begin{equation}\label{csteq} {\rm(I)}\quad d {\phi}_i =
\sum_{j=1}^{n} {\phi}_{ij}\wedge {\phi}_j,\qquad {\rm(II)}\quad d
{\phi}_{jk} = \sum_{l=1}^{n} {\phi}_{jl}\wedge {\phi}_{lk} -
\Phi_{jk}\end{equation} for any $i, j,
k=1,...,n$, where
${\phi}_{ij}(X):=\left\langle\nabla_{X} \tau_j, \tau_i\right\rangle$ denote the
{\it connection 1-forms} of $\underline{\phi}$ and
$\Phi_{jk}$ denote the {\it curvature 2-forms}, defined by
$\Phi_{jk}(X,Y):={\phi}_k(\RC(X,Y)\tau_j)$ for any $X, Y\in\XG$, where
$\RC$  is the {\it Riemannian curvature
tensor}, that is
\begin{eqnarray*}\RC(X,Y)Z:=\nabla_Y\nabla_XZ-\nabla_X\nabla_YZ -\nabla_{[Y,X]}Z\end{eqnarray*}for any $X, Y, Z\in\XG$.
The following holds
\begin{equation}\label{fr}C_{ij}^{k}:=\left\langle[\tau_i,\tau_j],\tau_k\right\rangle={\phi}_{jk}(\tau_i)- {\phi}_{ik}(\tau_j)\qquad\forall\,\,i, j,
k=1,...,n.\end{equation}This identity can be proved by using the fact
that $\nabla$ is torsion-free.

\begin{Defi}\label{hyppl} A \rm vertical hyperplane $\mathcal I$ \it is the zero-set of a linear
homogeneous polynomial on $\GG$  of homogeneous degree $1$.
A \rm non-vertical hyperplane $\mathcal I$ \it is the zero-set of a linear
polynomial on $\GG$  of homogeneous degree greater than or equal to $2$. 
\end{Defi}
 Hyperplanes  are $(n-1)$-dimensional vector subspaces of $\gg$.
The importance of vertical hyperplanes comes from the intrinsic rectifiability theory developed by
Franchi, Serapioni and Serra Cassano in 2-step Carnot groups;
see \cite{FSSC3, FSSC4, FSSC5, FSSC6}. They turn out to be ideals
of the Lie algebra $\gg$ and may be thought of as generalized
tangent spaces to sets of finite $\HH$-perimeter (in the
variational sense); see  \cite{AKLD}.  
  Non-vertical hyperplanes will be studied in Section \ref{exf}.

\section{Divergence formulae}\label{HDF}
Let  $S\subset\GG$ be a  
hypersurface of class $\cont^2$. For the sake of simplicity, we first assume that $S$ is non-characteristic.
 Let
$\cont^i\ss(S),\,(i=1, 2)$ be the space of functions whose
$\HS$-derivatives up to the $i$-th order are continuous on $S$.  Analogously, for any open subset $\UU\subseteq S$, we set $\cont^i\ss(\UU)$, to denote the space of functions whose
$\HS$-derivatives up to the $i$-th order are continuous on $\UU$. The previous definitions extend to the case  $C_S\neq \emptyset$ by requiring that all
$\HS$-derivatives up to the $i$-th order are continuous on $C_S$. \begin{oss}
The notions concerning the $\HS$-connection $\gs$,  the horizontal 2nd fundamental form $B\cc$ and the torsion $\Tor\ss$, can also be  formulated  by replacing $\XX^1(\HS)=\cont^1(S, \HS)$ with the larger  space  $\XX\ss^1(\HS):=\cont\ss^1(S, \HS)$.
\end{oss}\begin{Defi}[$\HS$-differential operators]\label{Deflh}Let $\lh:\XX\ss^1(\HS)\longrightarrow\cont(S)$  be
the 1st order
differential operator given by
\begin{eqnarray*}\lh X:=\div\ss X +  \left\langle C\cc(\varpi\cd)\nn,
X\right\rangle \qquad
\forall\,\,X\in\XX\ss^1(\HS).\end{eqnarray*}Moreover, let $\lh:\cont\ss^2(S)\longrightarrow\cont(S)$ be the 2nd order differential operator defined as
\begin{eqnarray*}\lll\varphi:=\Delta\ss\varphi +  \left\langle C\cc(\varpi\cd)\nn,
\qq\varphi\right\rangle
\qquad\forall\,\,\varphi\in\cont^2\ss(S).\end{eqnarray*}
\end{Defi}Note that $\lh(\varphi X)=\varphi\lh X +
\left\langle\qq \varphi, X\right\rangle$   for every $X\in\XX^1(\HS)$ and every $\varphi\in\cont\ss^1(S)$. Moreover  $\lll \varphi=\lh(\qq
\varphi)$ for every $\varphi\in\cont\ss^2(S)$.

It is not difficult to see that the operators $\Delta\ss$ and $\lll$ naturally
extend to  horizontal vector fields. These extensions will be denoted by $\overrightarrow{\Delta\ss}$ and $\overrightarrow{\lll}$. We remark that$$\overrightarrow{\lll}X=\overrightarrow{\Delta\ss}X+\left({\mathcal J}\ss X\right)C\cc(\varpi\cd)\nn$$ for every $X\in\cont^2\ss(S\setminus C_S, \HS)$, where ${\mathcal J}\ss X$ denotes the $\HS$-Jacobian matrix of the horizontal tangent vector field $X$.

We now define a homogeneous measure $\nis$, which plays the role of the intrinsic Hausdorff measure on $(n-2)$-dimensional submanifolds of $\GG$.

\begin{oss}[The measure $\nis$]\label{oas}Let $\eta\in\XX(\TS)$ be a unit normal vector orienting $\partial S$. Furthermore, let $\eta\ss:=\frac{\P\ss\eta}{|\P\ss\eta|}$ be
 the {\rm unit
$\HS$-normal} of $\partial S$. By definition, we set
 ${\nis}:=
 \left(\eta\ss\LL\per\right)\big|_{\partial S}$. Exactly as for the $\HH$-perimeter $\per$, the measure $\nis$, which is $(Q-2)$-homogeneous with respect to Carnot dilations, can be
represented in terms of the Riemannian measure ${\sigma^{n-2}\rr}$. In fact, we have $\nis =
|\PH\nu|\,|\P\ss\eta|\,\sigma^{n-2}\rr\res{\partial S}$.
\end{oss}

The above definitions are somehow motivated by
Theorem 3.17 in \cite{Monteb}. 
\begin{teo}\label{GD}Let $S\subset\GG$ be a  
compact non-characteristic hypersurface of class $\cont^2$ with (piecewise)
$\cont^1$ boundary $\partial S$ and $X\in\XX^1(\TS)$. Then
$$\int_{S}\lh X\,\per= 
\int_{\partial S}\left\langle
X,\eta\ss\right\rangle\,\nis.$$\end{teo}

As a consequence, the following integral formula holds
\begin{equation}\label{fipp}
 \int_{S}\lh X\,\per=\int_{S}\left( \div\ss X + \left\langle C\cc(\varpi\cd)\nn,
X\right\rangle\right)\,\per = 0
\end{equation}for every $X\in
\cont^1_0(S, \HS)$.

 Stokes' formula is concerned with integrating a $k$-form over a $k$-dimensional manifold with boundary.
A common way to state this fundamental result is the following.

\begin{Prop}\label{ST} Let $M$ be an oriented $k$-dimensional manifold of class $\cont^2$ with boundary
$\partial M$. Then $\int_M d\alpha=\int_{\partial M}\alpha$ for every compactly supported $(k-1)$-form $\alpha$ of class $\cont^1$.
\end{Prop}
One requires $M$ to be of class $\cont^2$ for a technical reason concerning \textquotedblleft pull-back\textquotedblright\, of differential forms.
We remark that it is possible to extend Proposition \ref{ST}
to the following cases:\begin{itemize}\item[$(\star)$] \it $\overline{M}$ is of class $\cont^1$ and $\alpha$ is a $(k-1)$-form such that $\alpha$ and $d\alpha$ are continuous;\item[$(\spadesuit)$] \it $\overline{M}$ is of class $\cont^1$ and $\alpha$ is a $(k-1)$-form such that $\alpha\in L^\infty(M)$, $d\alpha\in L^1(M)$ -or  $d\alpha\in L^\infty(M)$- and $\imath_M^\ast\alpha\in L^\infty(\partial M)$, where $\imath_M:\partial M\longrightarrow\overline{M}$ is the natural inclusion. 
\end{itemize}
 
 \rm Many different versions of Stokes' theorem are available in literature; see, for instance, \cite{FE}. For an introduction, we refer the reader to the book by Taylor \cite{Taylor}; see Appendix G. 

\begin{oss}General versions of  Stokes' theorem  can be deduced from the generalized Gauss-Green formula 
proved by De Giorgi and Federer; see  \cite{DeGio} or \cite{FE}, Theorem 4.5.6, p. 478. However, it is worth observing that they hold for Lipschitz differential forms. A general result of this type can be found in Maz'ja'  \cite{Maz}, see Section 6.2. We observe that  $(\star)$  holds for any compact oriented  $k$-dimensional manifold $M$ with boundary and for differential forms that are Lipschitz at each point of $M\setminus T$, where $T$ is a thin subset of $M$; see \cite{Pfeffer}, Remark 5.3.2, p. 197.
On the other hand $(\spadesuit)$ is perhaps less known that $(\star)$. 
The validity of $(\spadesuit)$ is observed in \cite{Taylor}; see formula (G.38), Appendix G. This result can be deduced by applying a standard procedure\footnote{See, for instance, Federer \cite{FE}, paragraph 3.2.46, p. 280; see also \cite{Pfeffer}, Remark 5.3.2, p. 197.} from a divergence-type theorem proved by  Anzellotti; see, more precisely, Theorem 1.9 in \cite{Anze}.
More recent and  more general results can also  be found in the  paper by Chen, Torres and Ziemer  \cite{variZ}.
\end{oss}

We have here to remark that either  condition $(\star)$  or $(\spadesuit)$ can  be used  to extend the horizontal integration by parts  formulae to vector fields (and functions)  possibly singular at the characteristic set $C_S$.  

\begin{Defi}\label{adm}
Let $X\in\cont^1(S\setminus C_S, \HS)$ and set
 $\alpha_X:=(X\LL \per)|_S$. We say that $X$ is \rm admissible (for the horizontal divergence formula)  \it if the differential forms $\alpha_X$ and $d\alpha_X$ satisfy either  condition $(\star)$  or $(\spadesuit)$ on  $S$. We say that $\phi\in\cont^2\ss(S\setminus C_S)$  is \rm admissible \it if $\qq \phi$ is admissible for the horizontal divergence formula. More generally, let $X\in\cont^1(S\setminus C_S, \TS)$ and  set
 $\alpha_X:=(X\LL \per)|_S$. Then,  we say that $X$ is \rm admissible (for the Riemannian divergence formula)  \it whenever  $\alpha_X$ and $d\alpha_X$ satisfy either  condition $(\star)$  or $(\spadesuit)$ on  $S$.
\end{Defi}
Using Definition \ref{adm} and Theorem \ref{GD} yields the following:
\begin{corollario}\label{GDc}Let $S\subset\GG$ be a 
compact hypersurface of class $\cont^2$ with  
$\cont^1$ boundary $\partial S$. We have:\begin{itemize}
 \item[{\rm (i)}]$\int_{S}\lh X\,\per=
\int_{\partial S}\left\langle
X,\eta\ss\right\rangle\,\nis$ for every  admissible   $X\in\cont^1(S\setminus C_S, \HS)$;
 \item[{\rm (ii)}] $\int_{S}\lll \phi\,\per=
\int_{\partial S}\left\langle
\qq\phi,\eta\ss\right\rangle\,\nis$ for every  admissible   $\phi \in\cont^2\ss(S\setminus C_S)$; \item[{\rm (iii)}] if $\partial S=\emptyset$, then  \begin{equation}\label{eq}-\int_{S}\varphi\lll
\varphi\,\per=\int_{S}|\qq\varphi|^2\,\per\end{equation} for every function $\varphi\in\cont^2\ss(S\setminus C_S)$ such that $\varphi^2$ is admissible. \end{itemize}
\end{corollario}
Note that formula \ref{eq} holds  even if $\partial S\neq\emptyset$, but in this case we have to use compactly supported functions on $S$.

\begin{oss}\label{remdiff}Let $\varphi\in\cont^2\ss(S\setminus C_S)$. If   $\varphi^2$ is admissible, then $$\varphi\in W\ss^{1,2}(S,\per)=\left\{\varphi\in L^2(S,\per): |\qq\varphi|\in L^2(S,\per)\right\}.$$ 
\end{oss}

\begin{es}[Heisenberg group; see Example \ref{hng2}]\label{hng3} One has
$\lh(X):=\div\ss X +\varpi\left\langle C\cc^{2n+1}\nn, X\right\rangle$ for every $X\in\XX^1(\HS)$ and $\lll\varphi:=\lh(\qq \varphi)=\Delta\ss\varphi +\varpi\langle C\cc^{2n+1}\nn,\qq\varphi\rangle$
 for every $\varphi\in\cont^2\ss(S)$.
\end{es}
 
The following result will be used throughout the proof of Corollary \ref{corvar2stat}.
\begin{Prop}\label{divtoeRiemW12} Let $M$ be a compact oriented  Lipschitz $k$-dimensional Riemannian manifold. Then $\int_M \div\tm X\,\sigma\rr^{k}=0$ for every $X\in W^{1, 2}_{comp}(M, \TT M)$, where $\sigma\rr^{k}$ is the Riemannian volume and $W^{1, 2}_{comp}(M)$ denotes the Sobolev space of all square-integrable compactly supported vector fields on $M$  with  square-integrable first (covariant) derivative.
\end{Prop}
This propoposition can be found in \cite{Shubin}; see Proposition 3.1, p. 7; see also \cite{schick}, Section 4.3, p. 33.

\subsection{Preliminary remarks concerning the 1st variation formula}\label{sec2.1}

\begin{no}Let $S\subset\GG$ be a  hypersurface of class  $\cont^i$ , $i\geq 2$. Let $X\in\TG$ and let $\nu$ be the  outward-pointing unit normal vector along $S$. We shall denote by $X\op$ and $X\ot$ the standard decomposition of $X$ into its normal and tangential components, i.e. $X\op=\langle X,\nu\rangle\nu,\,X\ot=X-X\op$. 
\end{no}
We now make a simple (but fundamental) calculation. \begin{lemma}\label{fondam}If $X\in\XX^1(\TG)$, then   $(X\LL\per)|_S=\left( \left(X\ot|\PH\nu|-\langle X,\nu\rangle\nn\ot\right)\LL\sigma\rr^{n-1}\right) \res S$. Moreover, at each non-characteristic point of $S$, we have $$d(X\LL\per)|_S=\div\ts \left( X\ot|\PH\nu|-\langle X,\nu\rangle\nn\ot \right) \,\sigma\rr^{n-1}\res S.$$
\end{lemma}
\begin{proof}We have
\begin{eqnarray*}d(X\LL\per)|_S &=& (X\LL\nn\LL\sigma\rr^n)|_S\\&=&d\left( \left( X\ot + X\op\right)\LL \left( \nn\ot + \nn\op\right)\LL\sigma\rr^n\right)\big|_S\\&=&d\left( X\ot \LL  \nn\op \LL\sigma\rr^n\right)\big|_S+
d\left( \nn\ot\LL X\op \LL\sigma\rr^n\right)\big|_S\\&=&d \left(X\ot\LL\per\right)\big|_S+d \left(\nn\ot\LL\langle X,\nu\rangle\sigma\rr^{n-1}\right)\big|_S\\&=&\div\ts \left( X\ot|\PH\nu|-\langle X,\nu\rangle\nn\ot \right) \,\sigma\rr^{n-1}\res S.\end{eqnarray*}
\end{proof}
 \begin{oss}\label{mistake} The previous calculation corrects a mistake in \cite{Monteb}, where the normal component of the vector field $X$
 was omitted. This has caused the loss of some divergence-type terms in some of the variational formulae proved there. 
\end{oss}
We would like to stress that the importance of the previous  calculation in the development of this paper comes from the well-known Cartan's identity for the Lie derivative of a differential form; see \cite{Ch1}, \cite{Lee}. More precisely, let $M$ be  a smooth manifold, let $\omega\in\Omega^k(M)$ be a differential  $k$-form on $M$ and let $X\in\XX(\TT M)$ be a differentiable vector field on $M$, with associated flow $\phi_t:M\longrightarrow M$. We recall that the Lie derivative of $\omega$  with respect to $X$, is defined by $\Lie_X\omega:=\frac{d}{dt}\phi_t^\ast\omega\big|_{t=0},$ where $\phi_t^\ast\omega$ denotes the pull-back of $\omega$ by $\phi_t$. Then, Cartan's identity says that \begin{equation}\label{ld}\Lie_X\omega= (X\LL d\omega) +d(X\LL\omega).\end{equation}This  is a very useful tool in proving variational formulae, not only for the case of  Riemannian volume forms, for which we refer the reader  to Spivak's book \cite{Spiv} (see Ch. 9, pp. 411-426 and 513-535), but even for more general functionals; see, for instance, \cite{Her}, \cite{Griff}. In Section \ref{varformsec}, we shall apply this method to write down the 1st and 2nd variation formulae for the $\HH$-perimeter measure $\per$.
But let us say something more about the 1st variation formula. 
So let $S\subset\GG$ be a hypersurface of class $\cont^2$. We remark that the Lie derivative of $\per$ with respect to $X$ can be calculated elementarily as follows. 
We begin with the first term in formula \eqref{ld}. We have$$ X\LL d \per=X\LL d(\nn\LL\sigma\rr^n)= X\LL\left(\div\,\nn \sigma\rr^n\right)=\langle X,\nu\rangle\,\div\,\nn\,\sigma\rr^{n-1}.$$Note that $\div\,\nn =\div\cc\nn=-\MS$. More precisely $$\div\,\nn =\sum_{i=1}^n\langle\nabla_{X_i}\nn,
{X_i}\rangle=\sum_{i=1}^\DH X_i({\nn}_i)=\div\cc\nn=-\MS.$$The second term in formula \eqref{ld} has been already computed in Lemma \ref{fondam}. Thus, we can  conclude that
\begin{equation}\label{9}
 \Lie_X\per=\left(-\MS\langle X,\nu\rangle +\div\ts \left( X\ot|\PH\nu|-\langle X,\nu\rangle\nn\ot \right)\right)\,\sigma\rr^{n-1},
\end{equation}at each non-characteristic point of $S$. We will return on this point in Section \ref{varformsec}.
\begin{oss}\label{99}Roughly speaking,  formula \eqref{9}   gives the \textquotedblleft infinitesimal\textquotedblright 1st variation of the measure $\per$. However,  in order to integrate the function $\Lie_X\per$ over a $\cont^2$  hypersurface  $S$,  we have to require that $\MS$ is locally integrable on $S$, with respect to the Riemannian measure $\sigma\rr^{n-1}$, i.e. 
\begin{equation}\label{Hint}\MS\in L^1_{loc}(S; \sigma\rr^{n-1}). 
\end{equation}Indeed, in general, $\MS$ fails to be integrable locally around the characteristic $C_S$; see, for instance, \cite{gar2}. Note that \eqref{Hint} implies the integrability of the function $\Lie_X\per$. Clearly, if $C_S=\emptyset$, then  \eqref{Hint} is automatically satisfied because, if $S$ is of class $\cont^2$, then $\MS\in\cont(S)$.
\end{oss}

\begin{oss}[Riemannian case]\label{C1} We would like to stress the analogy with the 1st variation of $\sigma\rr^{n-1}$ for a hypersurface $S$ of class $\cont^i$, $i\geq 1$, immersed in the Euclidean space $\Rn$. It is well-known that the 1st variation formula is given by $I_S(\sigma\rr^{n-1})=\int_S\div\ts W\,\sigma\rr^{n-1};$ see Simon's book \cite{Simon}, Ch. 2, $\S$ 9, pp. 48-53. In the $\cont^1$-case,  the variation vector $W$  cannot be decomposed  in its normal and tangential parts.  Clearly, this can be done if $S$ is of class $\cont^2$. In this case $$I_S(\sigma\rr^{n-1})=\int_S\div\ts W\,\sigma\rr^{n-1}=\int_S\left( \langle W\op,\nu\rangle\,\div\ts \nu+\div\ts W\ot\right) \,\sigma\rr^{n-1}.$$Note that $-\mathcal H\rr=\div\ts\nu$. Hence, we have two contributions. The first is given by $-\int_S \mathcal H\rr\langle W\op,\nu\rangle\,\sigma\rr^{n-1}$ and only depends on the normal component of the variation vector $W$. By means of the divergence theorem, the second term can be transformed in a boundary integral\footnote{In this case, we further assume that $\partial S$ is a $(n-2)$-dimensional submanifold of class $\cont^1$ oriented by the outward unit normal vector $\eta$.} given by $\int_{\partial S}\langle W\ot,\eta\rangle\,\sigma\rr^{n-2}$ which really depends only on the tangential component of $W$. 
\end{oss}

\

\begin{oss}[Horizontal variations]Let $S\subset\GG$ be a compact hypersurface of class $\cont^2$ and $\nu$ the outward-pointing unit normal vector along $S$. We  observe that formula \eqref{fipp} generalizes to the following:\begin{eqnarray}\label{opophparts} \int_{S}\lh X\,\per=-\int_{S}\MS\langle X, \nn\rangle\,\per +
\int_{\partial S}\langle
X,\eta\ss\rangle\,\nis\qquad \forall\,\,X\in\XX^1(\HH);\end{eqnarray}see, for instance, Corollary 3.19 and Theorem 4.3 in \cite{Monteb}.  Note that \eqref{opophparts} holds if either $C_S=\emptyset$ or, if $C_S\neq \emptyset$,   whenever $X\in\cont^1(S\setminus C_S, \HH)$ is an admissible vector field.
\end{oss}

Formula \eqref{opophparts} can be seen as a particular case of the  1st variation formula of the $\HH$-perimeter; see formula \eqref{fva} in Theorem \ref{1vg} below.
Actually, note that if  $X=X\cc\in\XH$, then
\begin{eqnarray*}X\cc\ot|\PH\nu|-\langle X\cc\op,\nu\rangle\nn\ot&=&\left(X\cc-|\PH\nu|\langle X\cc,\nn\rangle\nu\right) |\PH\nu|-|\PH\nu|\langle X\cc ,\nu\rangle\left(\nn-|\PH\nu|\nu\right)\\&=&\left(X\cc-\langle X\cc ,\nu\rangle\nn\right)|\PH\nu|\quad\Big(=:X\ss|\PH\nu|\Big),\end{eqnarray*}where we have used the fact that $\nu=|\PH\nu|\nn+\sum_{\alpha\in I\vv}\nu_\alpha X_\alpha$ at each non-characteristic point. Now inserting this  into  \eqref{9} yields 
\begin{eqnarray*} 
 \Lie_{X\cc}\per&=&\left(-\MS\langle X\cc,\nu\rangle +\div\ts\left(X\ss|\PH\nu|\right)\right)\,\sigma\rr^{n-1}\\&=& -\MS\langle X\cc,\nn\rangle\,\per +\div\ts\left(X\ss|\PH\nu|\right).
\end{eqnarray*}Hence,  integrating this expression along $S$ and  using Corollary \ref{GDc}, the  claim follows.

\section{Some technical preliminaries about the connection 1-forms}\label{tecsec1}
 
 Let $S\subset\GG$ be a 
hypersurface of class $\cont^2$ and let $U\subset\GG$ be an open set having non-empty
intersection with $S$ and such that $\UU:=U\cap S$ is
non-characteristic. We start with an elementary calculation.

\begin{lemma}\label{inlemma} We have $\div\ts\nn=-\MS-\left\langle C(\P\vv)\nn, \P\vv\nu\right\rangle$, where $ C(\P\vv):=\sum_{\alpha\in I\vv}\nu_\alpha C^\alpha$.
\end{lemma}

\begin{proof}We have $\div\ts\nn=\div\,\nn-\langle\nabla_{\nu}\nn,\nu\rangle$. Since $\div\,\nn= -\MS$, the thesis follows from$$\langle\nabla_{\nu}\nn,\nu\rangle=\sum_{j\in I\cc}\sum_{\alpha, \beta\in I\vv}{\nn}_j\nu_\alpha\nu_\beta\langle\nabla_{X_\alpha}X_j,X_\beta\rangle=\sum_{j\in I\cc}\sum_{\alpha, \beta\in I\vv}{\nn}_j\nu_\alpha\nu_\beta\dfrac{\left(\SC_{\alpha j}^\beta+\SC_{\beta j}^\alpha\right)}{2}= \left\langle C(\P\vv\nu)\nn, \P\vv\nu\right\rangle.$$\end{proof}
\begin{oss}\label{remfac}We have
   \begin{equation}\label{gghh}
    -\MS=\div\cc\nn=\div\cc\left(\frac{\PH\nu}{|\PH\nu|}\right)
  =\frac{\div\cc(\PH\nu)-\left\langle\grad\cc|\PH\nu|, \nn\right\rangle}{|\PH\nu|}.\end{equation} Since $|\PH\nu|$ is Lipschitz
continuous, it follows that $\MS\in L_{loc}^1(S; \per)$, but not necessarily $L_{loc}^1(S; \sigma\rr^{n-1})$. Note also that the last condition follows by assuming $\frac{1}{|\PH\nu|}\in L_{loc}^1(S; \sigma\rr^{n-1})$.\end{oss}

\begin{lemma}\label{primavolta}The following identities hold:

\begin{itemize}

\item[{\rm (i)}] $\phi_{1i}(\tau_j)=\phi_{1j}(\tau_i) + \left\langle
C\cc(\varpi\cd)\tau_i,\tau_j\right\rangle\quad\forall\,\,i, j\in I\ss$;  \item[{\rm
(ii)}]$\phi_{1i}(\TB_\alpha)
=\tau_i(\varpi_{\alpha})+\frac{1}{2}\left\langle
C^\alpha\cc\tau_1,\tau_i\right\rangle - \left\langle
C(\varpi)\TB_\alpha,\tau_i\right\rangle\quad\forall\,\,i\in I\cc\,\,\forall\,\,\alpha\in I\vv;$

\item[{\rm (iii)}]$\phi_{i\alpha}(\tau_j)=\phi_{j\alpha}(\tau_i) +
\left\langle C\cc^\alpha\tau_i,\tau_j\right\rangle\quad\forall\,\,i, j\in I\cc \, \,\forall\,\,
\alpha\in I\vv;$

\item[{\rm (iv)}]$\TB_\alpha (\varpi_\beta)- \TB_\beta
(\varpi_\alpha) = \left\langle
C(\varpi)\TB_\beta,\TB_\alpha\right\rangle\quad\forall\,\,\alpha,\beta \in I\vv$;

\item[{\rm (v)}] $\phi_{i\alpha}(\tau_\alpha)=0\quad\forall\,\,i\in
I\cc\,\, \forall\,\,\alpha\in I\vv;$ \item[{\rm (vi)}] $\phi_{\alpha
i}(\tau_i)=0\quad\forall\,\,i\in I\cc\,\,\forall\,\,\alpha\in I\vv$;\item[{\rm (vii)}]
$\phi_{i\alpha}(\tau_j)=\frac{1}{2}\left\langle
C^\alpha\cc\tau_i,\tau_j\right\rangle\quad\forall\,\,i, j\in I\cc\,\,\forall\,\,\alpha\in
I\vv$.

\end{itemize}

\end{lemma}\begin{proof}By direct computation. In particular, using the fact that the Lie brackets of tangent vector fields along $S$ is still tangent; for a detailed proof, see \cite{Monteb}.
\end{proof}

\begin{lemma}\label{yuiq}The matrix of the linear operator $B\cc$ can be written out as a sum of two
matrices, one symmetric and the other skew-symmetric, i.e. $B\cc=
S\cc + A\cc,$ where the skew-symmetric matrix
 $A\cc$ is  given by  $A\cc=\frac{1}{2}\,C\cc(\varpi\cd)|\ss$.
\end{lemma}
\begin{proof}It is sufficient to apply (i) of Lemma \ref{primavolta}.
\end{proof}

\begin{lemma}\label{kl} One has ${\rm Tr}\left(B^2\cc\right)=\|S\cc\|\ngr^2-\|A\cc\|\ngr^2=\sum_{j, k\in I\ss} \phi_{1k}(\tau_j)\phi_{1j}(\tau_k).$
\end{lemma}\begin{proof}
We have \begin{eqnarray*}\sum_{j, k\in I\ss} \phi_{1k}(\tau_j)\phi_{1j}(\tau_k)&=&\sum_{j, k\in I\ss} \left\langle\nabla_{\tau_j}\tau_1,\tau_k\right\rangle \left\langle\nabla_{\tau_k}\tau_1,\tau_j\right\rangle  \\&=&\sum_{j, k\in I\ss} (B\cc)_{kj}(B\cc)_{jk}\\&=&{\rm Tr}\left(B^2\cc\right)\\&=&\sum_{j \in I\ss}\left\langle B\cc \tau_j, B\cc^{\rm Tr}\tau_j\right\rangle
\\&=&\sum_{j \in I\ss}\left\langle \left(S\cc +A\cc\right)\tau_j, \left(S\cc -A\cc\right)\tau_j\right\rangle\\&=&\|S\cc\|\ngr^2-\|A\cc\|\ngr^2.\end{eqnarray*}
\end{proof}

\begin{lemma}\label{mk} We have $\sum_{\alpha \in I\vv}\varpi_\alpha\lh\left(C\cc^\alpha\tau_1\right)=2 \|A\cc\|\ngr^2+ |C\cc(\varpi\cd)\tau_1|^2$.
\end{lemma}

\begin{proof} We have \begin{eqnarray*}\lh\left(C\cc^\alpha\tau_1\right)&=&\sum_{j\in I\ss}\left\langle\nabla_{\tau_j}C\cc^\alpha\tau_1,\tau_j\right\rangle+\left\langle C\cc^\alpha\tau_1,C\cc(\varpi\cd)\tau_1\right\rangle\\&=&-\sum_{j\in I\ss}\left\langle\nabla_{\tau_j}\tau_1,C\cc^\alpha\tau_j\right\rangle+\left\langle C\cc^\alpha\tau_1,C\cc(\varpi\cd)\tau_1\right\rangle\quad\mbox{(by linearity and skew-symmetry)}\\&=&-\sum_{j\in I\ss}\left\langle\nabla_{\tau_j}\tau_1,C\ss^\alpha\tau_j\right\rangle+\left\langle C\cc^\alpha\tau_1,C\cc(\varpi\cd)\tau_1\right\rangle,\end{eqnarray*}where $C\ss^\alpha:=C\cc^\alpha|_{\HS}$. Since
$\left\langle\nabla_{\tau_j}\tau_1,C\ss^\alpha\tau_j\right\rangle=-B\cc(\tau_j, C\ss^\alpha\tau_j)\,\,\forall\,j\in I\ss$, it follows that

\begin{eqnarray*}\sum_{\alpha \in I\vv}\varpi_\alpha\lh\left(C\cc^\alpha\tau_1\right)&=&\sum_{\alpha \in I\vv}\varpi_\alpha\sum_{j\in I\ss}B\cc(\tau_j, C\ss^\alpha\tau_j)+|C\cc(\varpi\cd)\tau_1|^2\\&=& \varpi_\alpha\sum_{j\in I\ss}B\cc(\tau_j, C\ss(\varpi\cd)\tau_j)+ |C\cc(\varpi\cd)\tau_1|^2,\end{eqnarray*}where $C\ss(\varpi\cd)=C\cc(\varpi\cd)|_{\HS}=2 A\cc$; see Lemma \ref{yuiq}. Therefore

\begin{eqnarray*}\sum_{\alpha \in I\vv}\varpi_\alpha\lh\left(C\cc^\alpha\tau_1\right)&=&  2\sum_{j\in I\ss}B\cc(\tau_j, A\cc\tau_j)+ |C\cc(\varpi\cd)\tau_1|^2\\&=& 2\sum_{j\in I\ss}\left\langle(S\cc+A\cc) \tau_j, A\cc\tau_j\right\rangle+ |C\cc(\varpi\cd)\tau_1|^2\\&=& 2 \|A\cc\|\ngr^2+ |C\cc(\varpi\cd)\tau_1|^2,\end{eqnarray*}where we have used the elementary identity $\sum_{j\in I\ss}\left\langle S\cc \tau_j, A\cc\tau_j\right\rangle=0$. Let us prove the last identity.  For every $j\in
 I\ss$ one has
 \begin{eqnarray*}\left\langle S\cc\tau_j,A\cc\tau_j\right\rangle&=&
 \frac{1}{4}\left\langle \left(B\cc+B^{\rm Tr}\cc\right)\tau_j, \left(B\cc-B^{\rm Tr}\cc\right)\tau_j\right\rangle\\
 &=&\frac{1}{4}\left(\left\langle B\cc\tau_j, B\cc\tau_j\right\rangle -\left\langle B^{\rm Tr}\cc\tau_j,
 B^{\rm Tr}\cc\tau_j\right\rangle \right).\end{eqnarray*}By summing over $j\in I\ss$ we get that
$\mathrm{Tr}\left(S\cc(\,\cdot\,, A\cc\,
 \cdot)\right)=\|B\cc\|^2\ngr-\|B^{\rm Tr}\cc\|^2\ngr=0.$\end{proof}

We now recall some identities involving the (Riemannian) curvature
2-forms $\Phi_{IJ}$ associated with the o.n. co-frame
$\underline{\phi}$ (dual of $\underline{\tau}$) which can be found
in \cite{Monteb}. In particular, we need to calculate $\sum_{j\in I\ss}\Phi_{1j}(X,\tau_j)=\sum_{j\in
I\ss}\left\langle\RC(X,\tau_j)\tau_1,\tau_j\right\rangle$ for any $X\in\nn S$, which is nothing but the Ricci curvature for the
partial $\HS$-connection $\gs$.

\begin{lemma}\label{ricric}We have:
\begin{itemize}\item[{\rm(i)}]$\left\langle\RC(\tau_i,\tau_j)\tau_h,\tau_k\right\rangle=-\frac{3}{4}\sum_{\alpha\in
I\cd} \left\langle C\cc^\alpha\tau_i,\tau_j\right\rangle\left\langle
C\cc^\alpha\tau_h,\tau_k\right\rangle \qquad \forall\,\, i\,\, j, h, k\in
I\cc;$\item[{\rm (ii)}]
$\left\langle\RC(\tau_\beta,\tau_i)\tau_j,\tau_k\right\rangle=-\frac{1}{4}\sum_{\alpha\in
I\cd} \left\langle C\cc^\alpha\tau_j,\tau_k\right\rangle\left\langle
C^\beta\tau_\alpha,\tau_i\right\rangle\qquad \forall\,\,i, j, k\in I\cc,\,\beta\in
I\ctr.$\end{itemize}

\end{lemma}

\begin{lemma}\label{riccri}For every  $X =X\cc+
X\vv \in\XG$  transversal to $S$, i.e. $X\pitchfork S$, we have the formula
$$\sum_{j\in
I\ss}\Phi_{1j}(X,\tau_j)=-\frac{3}{4}\sum_{\alpha\in I\cd}\left\langle
C\cc^\alpha\nn,C\cc^\alpha X\cc\right\rangle-\frac{1}{4}\sum_{\alpha\in
I\cd}\sum_{\beta\in I\ctr} x_\beta \left\langle
C\cc^\alpha\nn,C^\beta\tau_\alpha\right\rangle.$$
\end{lemma}\begin{proof}Using Lemma \ref{ricric}.\end{proof}

\begin{lemma}\label{sple}Let $\underline{\tau}=\{\tau_1,...,\tau_n\}$ be an adapted o.n.
 frame for
$\UU\subseteq S$ on $U$ and fix $p_0\in\UU$. Then, we can always
  choose $\underline{\tau}$ so that the connection
1-forms
$\underline{\phi}=\{\phi_1,...,\phi_n\}$ satisfy
$\phi_{ij}(p_0)=0$ whenever $i, j\in I\ss=\{2,...,\DH\}$.
\end{lemma}
\begin{proof} The proof  follows by using a Riemannian geodesic frame.
So let  $\underline{\xi}=\{\xi_1,...,\xi_n\}$ be a o.n. frame on
$U$  adapted to $\UU=U\cap S$
satisfying $\xi_1(p)=\nu(p)$ and such that
$ \TT_pS=\mathrm{span}_\R\{\xi_2(p),...,\xi_n(p)\}$
for every $p\in\UU$. Let
$\underline{\varepsilon}=\{\varepsilon_1,...,\varepsilon_n\}$ denote its
dual co-frame.

\begin{claim} It is always possible to choose another
o.n.  frame $\underline{\widetilde{\xi}}$ on $U$
adapted to $\UU$ satisfying:
\begin{itemize}\item[{\rm
(i)}]$\underline{\widetilde{\xi}}(p_0)=\underline{\xi}(p_0)$;
\item[{\rm (ii)}]Let
$\widetilde{\varepsilon}_{ij}:=\left\langle\nabla{\widetilde{\xi}}_i,{\widetilde{\xi}}_j\right\rangle\,\,(i,j=1,...,n)$ denote the connection 1-forms of $\underline{\widetilde{\xi}}$. Then, one has
$\widetilde{\varepsilon}_{ij}(p_0)=0$ for every $i,j=2,...,n$.
\end{itemize}\noindent\end{claim}\noindent Clearly
$\underline{\widetilde{\xi}}=\{\widetilde{\xi}_2,...,\widetilde{\xi}_n\}$
is a tangent o.n. frame for $\UU$. The proof
of this claim is standard; see, for instance,
\cite{Spiv}, pag. 517-519, eq.(17). Now assuming that $\xi_i(p_0)=\tau_i(p_0)$ for every $i\in
I\ss$. In
particular, we have
$$\widetilde{\varepsilon}_{ij}(X_{p_0})=
\left\langle\nabla_{X_{p_0}}\widetilde{\xi}_i,\widetilde{\xi}_j\right\rangle(p_0)=0\qquad\forall\,\,i,j\in I\ss,\,\,\forall\,\,X\in\XX^1(\TS).$$By extending the o.n. frame
$\{\widetilde{\xi}_2,...,\widetilde{\xi}_{\DH}\}$ for the
horizontal tangent space $\HS$ to a full adapted frame
$\underline{\tau}$ in the sense of Definition \ref{movadafr}, the thesis easily follows.\end{proof}

The following notion will be used throughout the proof of Lemma \ref{vnct}.

\begin{Defi}\label{NDF} Let $S\subset\GG$ be a hypersurface of class  $\cont^i\,(i\geq 2)$. We say that a $\cont^i$-smooth function $f:\GG\longrightarrow\R$ is a defining function for $S$ if $S=\{x\in\GG : f=0\}$ and $\grad f\neq 0$ for all $x\in S$.
Furthermore, we say that $f$ is a \rm normalized defining function for $S$ (abbreviated as NDF) \it if, and only if, $|\grad\cc f|=1$ for all $x\in S\setminus C_S$. \end{Defi}

\begin{oss}Some remarks are in order. First, it is not difficult to see that, given a defining function $f$ for $S$, then a NDF $\widetilde{f}$ for $S$ can simply be defined by dividing $f$ by the magnitude of its horizontal gradient $|\grad\cc f|$, i.e. $$\grad\widetilde{f}(p)=\grad\left( \frac{f}{|\grad\cc f|}\right)(p)=\frac{\grad f}{|\grad\cc f|}(p)=\nn(p)+\varpi(p)\qquad \forall\,\,p\in S\setminus C_S.$$Note that the  NDF $\widetilde{f}$  is one order of differentiability less smooth than $f$. This is what happens also in the Euclidean case; see the  book by Krantz and Parks \cite{Krantz} and references therein. However, at least for $2$-step Carnot groups, a normalized defining function of class $\cont^i$ for every $\cont^i$-smooth hypersurface $S$ $(i\geq 2)$, is given by the (signed) CC-distance function from $S$; see \cite{Montegv}.
\end{oss}

We end this section with a  lemma, which will be important in the sequel. Let $S$ be as above, let $p_0\in S$ and assume that, locally around $p_0$, $S$ is the level set of a function $f:U\subset\GG\longrightarrow\R$. We  see that, locally around $p_0$, $X f=0$ for every $X\in \XX(\TS)$. In particular, $\TB_\alpha(f)=0$ for every $\alpha\in I\vv$. As a consequence,  by using an adapted frame $\underline{\tau}$, one has  $\tau_\alpha(f)=\varpi_\alpha\tau_1(f)$ for every $\alpha\in I\vv$.
A normal  vector along $S$ in a neighborhood of $p_0$ is given by  $\mathcal N:=\tau_1 (f) \tau_1+ \sum_{\alpha\in I\vv}\tau_\alpha(f)\tau_\alpha$ and we have $\nu=\frac{\mathcal N}{|\mathcal N|}$.
\begin{lemma}\label{xzc} The following identities hold:

\begin{itemize}

\item[{\rm (i)}] $\phi_{1j}(\tau_1)=\frac{\tau_j(\tau_1(f))}{\tau_1(f)}-  \left\langle C\cc(\varpi\cd)\tau_1,\tau_j\right\rangle\quad\forall\,\,j\in I\ss$;  \item[{\rm
(ii)}] $\phi_{1j}(\tau_\alpha)= \frac{1}{2}\left\langle
C^\alpha\cc\tau_1,\tau_j\right\rangle-  \left\langle C(\varpi)\tau_\alpha,\tau_j\right\rangle+ \frac{\tau_j(\tau_\alpha(f))}{\tau_1(f)}\quad\forall\,\,j\in I\ss\,\forall\,\,\alpha \in I\vv$.

\end{itemize}
\end{lemma}

\begin{proof}We have
$$[\tau_1, \tau_j]=\left\langle[\tau_1, \tau_j], \tau_1\right\rangle\tau_1 +\sum_{k\in I\ss}\left\langle[\tau_1, \tau_j], \tau_k\right\rangle\tau_k + \sum_{\alpha\in I\vv}\left\langle[\tau_1, \tau_j], \tau_\alpha\right\rangle\tau_\alpha. $$
Therefore $$[\tau_1, \tau_j](f)=-\tau_j(\tau_1(f))= \left\langle[\tau_1, \tau_j], \tau_1\right\rangle\tau_1(f)
+ \sum_{\alpha\in I\vv}\left\langle[\tau_1, \tau_j], \tau_\alpha\right\rangle\tau_\alpha(f)$$
and this implies that\begin{equation}C_{1j}^1=\phi_{1j}(\tau_1)=\left\langle[\tau_1, \tau_j], \tau_1\right\rangle=\dfrac{\tau_j(\tau_1(f))}{\tau_1(f)}-\sum_{\alpha\in I\vv}\dfrac{\tau_\alpha(f)}{\tau_1(f)}\left\langle C\cc^\alpha\tau_1,\tau_j\right\rangle,\end{equation}
where we have used the  identity $C_{1j}^\alpha=-\left\langle C\cc^\alpha\tau_1,\tau_j\right\rangle$. This proves (i).\\
In order to prove (ii), we compute
$$[\tau_\alpha, \tau_j]=\left\langle[\tau_\alpha, \tau_j], \tau_1\right\rangle\tau_1 +\sum_{k\in I\ss}\left\langle[\tau_\alpha, \tau_j], \tau_k\right\rangle\tau_k + \sum_{\beta \in I\vv}\left\langle[\tau_\alpha, \tau_j], \tau_ \beta\right\rangle\tau_\beta,$$from which we get
$$[\tau_\alpha, \tau_j](f)=-\tau_j(\tau_\alpha(f))= \left\langle[\tau_\alpha, \tau_j], \tau_1\right\rangle\tau_1(f)
+ \sum_{\beta\in I\vv}\left\langle[\tau_\alpha, \tau_j], \tau_\beta \right\rangle\tau_\beta(f).$$Thus
$$ -\dfrac{\tau_j(\tau_\alpha(f))}{\tau_1(f)}=-\phi_{1j}(\tau_\alpha)+ \phi_{1\alpha}(\tau_j)+ \sum_{\beta\in I\vv}\varpi_\beta C_{\alpha j}^1,$$ where we have used the  identity $C_{\alpha j}^1= \left\langle \nabla_{\tau_\alpha}\tau_j,\tau_1\right\rangle- \left\langle \nabla_{\tau_j}\tau_\alpha,\tau_1 \right\rangle$. Finally, since $\phi_{1\alpha}(\tau_j)=\frac{1}{2}\left\langle
C^\alpha\cc\tau_1,\tau_j\right\rangle$ (see (vii) of Lemma \ref{primavolta}), using $ C_{\alpha j}^\beta =-\left\langle C^\beta\tau_\alpha,\tau_j\right\rangle $ it follows that
\begin{equation}\phi_{1j}(\tau_\alpha)= \frac{1}{2}\left\langle
C^\alpha\cc\tau_1,\tau_j\right\rangle- \sum_{\beta\in I\vv}\varpi_\beta \left\langle C^\beta\tau_\alpha,\tau_j\right\rangle+ \dfrac{\tau_j(\tau_\alpha(f))}{\tau_1(f)},\end{equation}
 as wished.
\end{proof}

\section{Variational formulae for the $\HH$-perimeter $\per$}\label{varformsec}

 Below we will obtain the 1st and 2nd variation formulae for the $\HH$-perimeter measure $\per$. More precisely, we shall assume that $S\subset\GG$ is  of class $\cont^2$, for the 1st variation formula, and 
 that $S$ is  of class $\cont^3$ for the 2nd variation formula.  Under further hypotheses, our formulae allow  to move the characteristic set $C_S$ of $S$. 

We stress that, in the case of the first Heisenberg group $\mathbb H^1$, a 1st variation formula for characteristic surfaces of class $\cont^2$ was  obtained by Ritor\'{e} and Rosales in \cite{RR}. Furthermore, Hurtado, Ritor\'{e} and Rosales \cite{HRR} have proved  a formula for the 2nd variation of $\per$ that is very similar to that stated in Theorem \ref{2vg} below; see also the unpublished preprint \cite{HP2}, where similar results are stated in a general sub-Riemannian setting.

Let $S\subset\GG$ be a 
hypersurface of class $\cont^i\,(i=2,3)$, let $U\subset\GG$ be a relatively compact open set having non-empty
intersection with $S$ and set $\UU:=U\cap S$. 

The following calculations are made for a bounded open subset  $\UU$ of $S$. In particular, we  assume $\cont^1$-regularity of  $\partial \UU$. Clearly, if $S$ is a compact hypersurface with boundary,  the formulae obtained in the sequel will hold for $S$.

\begin{Defi} \label{leibniz}Let $\imath:\UU\rightarrow\GG$ denote the inclusion of $\UU\subset S$ in $\GG$
and let
 $\vartheta: ]-\epsilon,\epsilon[\times \UU
\rightarrow \GG$ be a ${\cont}^i$-smooth map, $i=2, 3$. We say that $\vartheta$ is a  {\rm
 variation} of $\imath$ if:
\begin{itemize}
\item[{\rm(i)}] every
$\vartheta_t:=\vartheta(t,\cdot):\UU\rightarrow\GG$ is an
immersion;\item[{\rm(ii)}] $\vartheta_0=\imath$.
\end{itemize}
Moreover, we say that  $\vartheta$ {\rm keeps the
boundary $\partial\UU$ fixed} if:\begin{itemize}
\item[{\rm(iii)}]$\vartheta_t|_{\partial \UU}=\imath|_{\partial
\mathcal{U}}$ for every $t\in ]-\epsilon,\epsilon[$.
\end{itemize}
The {\rm variation vector} of $\vartheta$ (i.e. its \textquotedblleft initial velocity\textquotedblright) is defined by
$W:=\frac{\partial \vartheta}{\partial
t}\big|_{t=0}=\vartheta_{\ast}\frac{\partial}{\partial
t}\big|_{t=0}.$
\end{Defi} We shall set
$\WW:=\frac{\partial\vartheta}{\partial
t}=\vartheta_{\ast}\frac{\partial}{\partial t}$ and assume that
$\WW$ is defined in a neighborhood of $\rm{Im}(\vartheta)$.  For
any \textquotedblleft time\textquotedblright\, $t\in]-\epsilon,\epsilon[$, let ${\nu}^t$ be the unit normal
vector along $\UU_{t}:=\vartheta_t(\UU)$ and let
$(\sigma^{n-1}\rr)_t$ be the Riemannian measure on $\UU_t$.
We assume that $f: U\longrightarrow \R$ is a local equation for the hypersurface
$S$ near $p_0\in S$ and that $f_t: ]-\epsilon,\epsilon[\times
U\longrightarrow \R$ is a family of $\cont^i$-smooth functions ($i=2,3$)
satisfying $f_0=f$ and $f_t(\vartheta_t(x))=t$ for every
$t\in]-\epsilon,\epsilon[$. In other words, the hypersurfaces
$\UU_t$ are level sets of a defining function $f_t$ and one has
$\left\langle\nabla f_t, \WW\right\rangle=1$. Choose
an o.n.  frame $\underline{\tau}$ on  $U\subset\GG$
satisfying:
\begin{equation}\tau_1|_{\UU_t}=\nt;\qquad
  \HH\TT_p\UU_t=\mathrm{span}\{(\tau_2)_p,...,(\tau_{\DH})_p\}\quad\forall\,\,p\in\UU_t;\qquad
 \tau_\alpha = X_\alpha \end{equation}for
every $t\in]-\epsilon,\epsilon[$. Furthermore, let
$\underline{\phi}:=\{\phi_1,...,\phi_n\}$ be the dual co-frame of $\underline{\tau}$ (i.e. $\phi_i(\tau_j)
 =\delta^j_i$ for all $i, j=1,...,n$). So, we have $\TB_\alpha f_t=0$; see Definition \ref{movadafr}.
This implies $\tau_\alpha(f_t)=\varpi^t_\alpha\tau_1(f_t)$, where $\varpi^t_\alpha:=\frac{\nu^t_\alpha}{|\PH\nu^t|}$.
Moreover, since $\left\langle\nabla f_t, \WW\right\rangle=1$, we have $\widetilde{w}_1\tau_1(f_t)+\sum_{\alpha\in I\vv}\ww_\alpha \tau_\alpha(f_t)=1,$ where $\widetilde{w}_1=\left\langle\WW, \tau_1\right\rangle$ and $\ww_\alpha=\left\langle\WW, \tau_\alpha\right\rangle$. Therefore
$$ \tau_1(f_t)\left(\ww_1+\sum_{\alpha\in I\vv}\ww_\alpha {\varpi^t}_\alpha\right)=1.$$Setting $\cn_t= \frac{\left\langle\WW, \nu^t\right\rangle}{|\PH \nu^t|}$ it follows that $\tau_1(f_t)=\dfrac{1}{\cn_t}$ and $\tau_\alpha(f_t)=\dfrac{\varpi^t_\alpha}{\cn_t}$.

The following technical result will be used in the proof of the 2nd variation of $\per$. 
\begin{lemma}\label{f}Under the previous assumptions, we have:
\begin{itemize}
 \item [{\rm (i)}]$\P\sst\left(\nabla_{\tau_1}\tau_1\right)=-\left(\frac{\grad\sst\cn_t}{\cn_t}+ C\cc(\varpi^t\cd)\tau_1\right)$;
\item [{\rm (ii)}]$\P\sst\left(\nabla_{\tau_\alpha}\tau_1\right) = \frac{1}{2}
C^\alpha\cc\tau_1 -   C(\varpi^t)\tau_\alpha +  {\grad\sst \varpi^t_\alpha }- \varpi^t_\alpha \frac{\grad\sst \cn_t }{\cn_t}\qquad\forall\,\,\alpha\in I\vv$.
\end{itemize}

\end{lemma}\begin{proof}By  applying (i) of Lemma \ref{xzc} we get that
$\phi_{1j}(\tau_1)=-\frac{\tau_j(\cn_t)}{\cn_t}-  \left\langle C\cc(\varpi^t\cd)\tau_1,\tau_j\right\rangle$. Furthermore,  (ii) of Lemma \ref{xzc} implies
\begin{equation}\label{b} \phi_{1j}(\tau_\alpha)= \frac{1}{2}\left\langle
C^\alpha\cc\tau_1,\tau_j\right\rangle-  \left\langle C(\varpi^t)\tau_\alpha,\tau_j\right\rangle+  {\tau_j( \varpi^t_\alpha)}- \varpi^t_\alpha \frac{\tau_j(\cn_t)}{\cn_t}\qquad\forall\,\,\alpha\in I\vv.
\end{equation}This achieves the proof.
 
\end{proof}
\noindent{\bf General remarks.}\,\,In order to discuss the variational formulae of $\per$, let us set
 $$\pert\res\,\UU_{\,t} = (\tau_1\LL
\phi_1\wedge...\wedge\phi_n)|_{\UU_{t}}=(\phi_2\wedge
...\wedge\phi_n)|_{\UU_{t}}.$$We also set $\Gamma(t):= \vartheta_t^\ast
(\phi_2\wedge...\wedge\phi_n).$ Note that $\Gamma: ]-\epsilon, \epsilon[\times\UU\longrightarrow\Omega^{n-1}(\UU)$ defines a 1-parameter family of differential $(n-1)$-forms on $\UU$.

\begin{oss}By definition, the \rm 1st and 2nd variation formulae of $\per$ along $\UU$ \it are given by
\begin{equation}\label{nome}I_\UU(\per):=\frac{d}{dt}\left(\int_{\UU}\Gamma(t)\right) \Bigg|_{t=0},\qquad
II_\UU(\per):=\frac{d^2}{dt^2}\left( \int_{\UU}\Gamma(t)\right) \Bigg|_{t=0}.\end{equation}
So we have a natural question: is it possible to bring the  \textquotedblleft time\textquotedblright derivatives inside the integral sign?
Note that  the answer is \textquotedblleft yes\textquotedblright  if we assume that $\overline{\UU}$ is non-characteristic. Indeed, in such a case it is not difficult\footnote{Actually, since $\dg f_t\neq 0$ at $t=0$, there must exist $\epsilon>0$ such that $\dg f_t\neq 0$ for all $t\in]-\epsilon, \epsilon[$ and hence $\nt=\frac{\grad\cc f_t}{|\grad\cc f_t|}$, which is the unit $\HH$-normal along $\UU_t=\vartheta_t(\UU)$, turns out to be of class $\cont^{i-1}$, $i=2, 3$. This implies that $\pert$ is $\cont^{i-1}$-smooth. Therefore $\Gamma(t)=\vartheta_t^\ast\pert$ is $\cont^{i-1}$-smooth.}  to show that there exists $\epsilon>0$ such that the 1-parameter family $\Gamma(\cdot)$ of differential $(n-1)$-forms on $\UU$ is $\cont^{i-1}$-smooth on $ ]-\epsilon, \epsilon[\times\UU$. This  allows us to estimate, uniformly in time,  both differential $(n-1)$-forms $\dot{\Gamma}(t)$ and $\ddot{\Gamma}(t)$.  We will return on this point later in this section.
\end{oss}

\vspace*{1cm}
 \begin{war}\label{lwar} Preliminarily, we  need the following assumptions:\begin{itemize}
 \item[{$\mathbf (A_1)$}]if $\UU$ is of class $\cont^2$  there exists an  integrable differential  $(n-1)$-form $\Phi_1\in\Omega^{n-1}(\UU)$, such that: \begin{equation*} 
 |\dot{\Gamma}(t)(\tt_1,...,\tt_{n-1})|\leq|\Phi_1(\tt_1,...,\tt_{n-1})|
\end{equation*} for every o.n. basis $\underline{\tt}=\{\tt_1,...,\tt_{n-1}\}$ of $\TT\UU$. 

 \item[{$\mathbf (A_2)$}]if $\UU$ is of class $\cont^3$  there exist   integrable differential  $(n-1)$-forms $\Phi_1, \Phi_2\in\Omega^{n-1}(\UU)$, such that: \begin{eqnarray*} 
 |\dot{\Gamma}(t)(\tt_1,...,\tt_{n-1})|&\leq&|\Phi_1(\tt_1,...,\tt_{n-1})|
\\ |\ddot{\Gamma}(t)(\tt_1,...,\tt_{n-1})|&\leq&|\Phi_2(\tt_1,...,\tt_{n-1})|\end{eqnarray*} for every o.n. basis $\underline{\tt}=\{\tt_1,...,\tt_{n-1}\}$ of $\TT\UU$. 
\end{itemize}
\end{war}

\noindent{\bf 1st variation.}\,\, We first note that
$$\int_{\UU}\Gamma(t)=\int_{\UU}\vartheta_t^\ast\pert=\int_{\UU}|\P\ct\nu^t|\,\mathcal{J}ac\,\vartheta_t\,\sigma\rr^{n-1},$$where $\mathcal{J}ac\,\vartheta_t $ denotes the usual Jacobian of the map $\vartheta_t$; see \cite{Simon}, Ch. 2, $\S$ 8, pp. 46-48. Indeed, by  definition, we have   $\pert=|\P\ct\nu^t|(\sigma\rr^{n-1})_t$ and hence the previous formula follows from the Area formula of Federer; see \cite{FE} or \cite{Simon}. Let us set $ f:]-\epsilon, \epsilon[\times\UU\longrightarrow\R$, \begin{equation}\label{faz} 
f(t, x):=|\P\ct\nu^t(x)|\,\mathcal{J}ac\,\vartheta_t(x).
\end{equation}In this case, we also set $C_{\UU}:=\left\lbrace x\in\UU: |\P\ct\nu^t(x)|=0 \right\rbrace$.  With this notation,  our original question can be solved by applying  to  $f$ the  Theorem of Differentiation under the integral; see \cite{Jost}, Corollary 1.2.2, p.124. More precisely, let us compute
\begin{eqnarray}\label{ujh}\frac{d f}{dt}&=&\frac{d\,|\P\ct\nu^t|}{dt}\,\mathcal{J}ac\,\vartheta_t + |\P\ct\nu^t|\frac{d\,\mathcal{J}ac\,\vartheta_t}{dt}\\\nonumber &=&\left\langle\WW,\grad\,|\P\ct\nu^t|\right\rangle\,\mathcal{J}ac\,\vartheta_t + |\P\ct\nu^t|\frac{d\,\mathcal{J}ac\,\vartheta_t }{dt}\\\nonumber &=&\left( \left\langle\WW\op,\grad\,|\P\ct\nu^t|\right\rangle+\left\langle\WW\ot,\grad\,|\P\ct\nu^t|\right\rangle + |\P\ct\nu^t|\div\tut\WW\right) \mathcal{J}ac\,\vartheta_t\\\nonumber &=&\left( \left\langle\WW\op,\grad\,|\P\ct\nu^t|\right\rangle+   \div\tut\left(\WW|\P\ct\nu^t|\right)\right) \mathcal{J}ac\,\vartheta_t,
\end{eqnarray}where we have used the very definition of tangential divergence and the well-known calculation of $\frac{d\,\mathcal{J}ac\,\vartheta_t}{dt}$, which can be found in Chavel's book \cite{Ch2}; see Ch.2, p.34. Now since  $|\P\ct\nu^t|$ is a Lipschitz continuous function, it follows that $\frac{d f}{dt}$ is bounded on $\UU\setminus C_{\UU}$ and so lies in $L^1_{loc}(\UU; \sigma\rr^{n-1})$. Therefore, we can pass the time-derivative  through the integral sign. 
This shows that: \it condition $(A_1)$ in Warning \ref{lwar} is always satisfied. \rm In particular, we have proved the following 1st variation formula:
\begin{equation}\label{iva}I_\UU(\per)=\int_{\UU}\dot{\Gamma}(0)=\int_{\UU}\left( \left\langle W\op,\grad\,|\P\cc\nn|\right\rangle+   \div\tu\left( W|\P\cc\nn|\right) \right)\,\sigma\rr^{n-1}.\end{equation}

 It follows from definitions that  $\frac{d f}{dt}$ can be regarded as the Lie derivative of 
 $\pert$ with respect to the variation vector $\WW$, that is
\begin{equation}\label{eqx}\frac{d f}{dt}=\vartheta_t^\ast\Lie_{\WW}\pert.\end{equation}

\begin{oss}Note that formula \eqref{eqx} can be proved exactly as in Spivak's book \cite{Spiv}, Ch. 9, p. 420. As already mentioned in  Section \ref{sec2.1}, this fact  allows us to use some standard tools in Differential Geometry such as the Cartan's  formula. In this way, another expression for the integrand $\dot{\Gamma}(0)$ can  be derived; see, for instance, formula \eqref{9}.  Nevertheless, this new expression it is not necessarily in $L^1_{loc}$, with respect to the Riemannian measure $\sigma\rr^{n-1}$; see Remark \ref{99}.\end{oss}
More precisely, we have  $$\dot{\Gamma}(0)=\imath^\ast\left(\mathcal{L}_{\WW} \pert\right)=
\imath^\ast\left(\mathcal{L}_{\WW}(\phi_2\wedge...\wedge\phi_n)\right).$$
By  Cartan's formula $$\mathcal{L}_{\WW} \pert = \WW \LL d\,\pert  +  d\,(\WW \LL \pert)$$ and hence
\begin{eqnarray}\label{verostatement} \dot{\Gamma}(0)=\imath^\ast
\left(\WW \LL d\,\pert  +  d\,\left(\WW \LL
\pert\right)\right).\end{eqnarray}By applying  the 1st structure equation of the co-frame $\underline{\phi}$ (see formula \eqref{csteq}) we have\begin{eqnarray}\nonumber
d\,\pert = 
\sum_{j=2}^{n}(-1)^{j}\phi_2\wedge...\wedge
d\,\phi_j\wedge...\wedge \phi_n = \sum_{j\in I\ss}\phi_{1j}(\tau_j)\,\phi_1\wedge...\wedge\phi_n=-\MST(\sigma\rr^n)_t,\end{eqnarray}
where  $\MST:=-\sum_{j\in
I\ss}\phi_{1j}(\tau_j)=\sum_{j\in
I\ss}\left\langle\gc_{\tau_j}\tau_j,\nt\right\rangle$ is the
horizontal mean curvature of $\UU_t$. Note that we  have used (v) of Lemma \ref{primavolta}.

The calculation of the second term has been discussed in detail in Section \ref{HDF}; see Lemma \ref{fondam}.
More precisely, we have $$d\,\left(\WW \LL \pert\right)=\div\tut\left( \WW \ot|\P\ct\nu^t|-\langle \WW ,\nu^t\rangle {\nt}\ot\right)\,(\sigma\rr^{n-1})_t.$$Therefore, under the previous assumptions, we have  \begin{equation}\label{kj} 
\Lie_{\WW}\pert=\left(-\MST\langle \WW,\nu^t\rangle +\div\tut \left(\WW\ot|\P\ct\nu^t|-\langle \WW,\nu^t\rangle{\nt}\ot \right)\right)\,(\sigma\rr^{n-1})_t.\end{equation}
Finally, the desired formula follows by setting $t=0$; see formula \eqref{9}.  
\begin{teo}[1st variation of $\per$] \label{1vg}Let $S\subset\GG$ be a compact 
hypersurface  of class $\cont^2$ with, or without, boundary and let
 $\vartheta: ]-\epsilon,\epsilon[\times S
\rightarrow \GG$ be a ${\cont}^2$-smooth variation of $S$. Let $W=\frac{d\,\vartheta_t}{dt}\big|_{t=0}$ be the variation vector field and let $W\op$ and $W\ot$ be the normal and  tangential components of $W$ along $S$, respectively. We also set $\cn:=\frac{\langle W\op,\nu\rangle}{|\P\cc\nu|}$. 
Then\begin{equation}\label{iva1}I_S(\per)=\int_{S}\left( \left\langle W\op,\grad\,|\P\cc\nn|\right\rangle+   \div\ts (W|\P\cc\nn|)\right)\,\sigma\rr^{n-1}.\end{equation}Furthermore, if $\MS\in L^1(S; \sigma\rr^{n-1})$, then  \begin{eqnarray}\label{fva}I_S(W,\per)& =& \int_{S} -\MS\cn\,\per +\int_S\div\ts\left( W\ot|\PH\nu|-\langle W,\nu\rangle\nn\ot \right) \sigma\rr^{n-1}\\\label{fvb} & =& 
\int_{S}\left(-\MS\langle W\op,\nu\rangle +\div\ts\left( W\ot|\PH\nu|-\langle W\op,\nu\rangle\nn\ot \right)\right)\,\sigma\rr^{n-1}
.\end{eqnarray}
\end{teo}
  
\begin{proof}Formula \eqref{iva1} is nothing but formula \eqref{iva}. Set  $t=0$ in formula \eqref{kj}.
If $\MS\in L^1(S; \sigma\rr^{n-1})$, then we can integrate this formula over  $S$.
Indeed, under such an assumption, all terms in the formula above turn out to be in $L^1(S; \sigma\rr^{n-1})$.  More precisely, for what concerns the  term $\div\ts\left( W\ot|\PH\nu|\right)$,   note that $W\ot\in\XX^1(\TS)=\cont^1(S,\TS)$ and  that $|\PH\nu|$ is Lipschitz continuous.   Moreover, if  $\MS\in L^1(S; \sigma\rr^{n-1})$,  the second term $\div\ts\left(\langle W,\nu\rangle\nn\ot\right)$ belongs to $L^1 (S; \sigma\rr^{n-1})$. In fact, one has$$\div\ts\left(\langle W,\nu\rangle\nn\ot\right)=\div\ts\left(\langle W,\nu\rangle\left(\nn-|\PH\nu|\nu\right)\right)$$and the claim  easily follows by using Lemma \ref{inlemma}. 
Hence, we have $$I_{S}(\per)=\int_{S}\dot{\Gamma}(0)=\int_{S}\mathcal{L}_{\WW} \pert\big|_{t=0}=\int_{S}\left(-\MS\langle W,\nu\rangle +\div\ts \left( W\ot|\PH\nu|-\langle W,\nu\rangle\nn\ot \right)\right)\,\sigma\rr^{n-1}.$$\end{proof}

\begin{corollario}\label{corvar1ma}Let the assumptions of Theorem \ref{1vg} hold. Let  $\partial S$ be of class $\cont^1$ and let $\eta$ be the unit outward normal along $\partial S$.  Then
 \begin{eqnarray}\label{corfva}I_S(W,\per)=
\int_{S} -\MS\cn\,\per +\int_{\partial S} \left\langle \left( W\ot|\PH\nu|-\langle W,\nu\rangle\nn\ot \right),\eta\right\rangle  \sigma\rr^{n-2}.\end{eqnarray}Furthermore, if $W$ is compactly supported on $S$, then \begin{eqnarray}\label{corfva}I_S(W,\per)=
\int_{S} -\MS\cn\,\per.\end{eqnarray}
\end{corollario}
\begin{proof}Immediate, since the vector field  $Y:=W\ot|\PH\nu|-\langle W\op,\nu\rangle\nn\ot$ is admissible (for the Riemannian divergence formula); see condition ($\spadesuit$) in Definition \ref{adm}. In fact, we see that $Y\in L^\infty(S)$ and using the fact that $\MS\in L^1(S; \sigma\rr^{n-1})$  yields $\div\ts(Y)\in L^1(S; \sigma\rr^{n-1})$.
 \end{proof}

\noindent{\bf 2nd variation.}\,\,We  regard this proof as a continuation of the proof of the 1st variation formula.
From now on, we  assume  $\UU$ and $S$ to be  of class $\cont^3$. Moreover, the boundary $\partial\UU$ (or, $\partial S$ when $S$ is compact) is assumed to be of class $\cont^1$. We also recall that, for the 2nd variation formula, the variation  $\vartheta$ is assumed to be of class $\cont^3$ on $]-\epsilon, \epsilon[\times\UU$. 

First, let us compute the second time-derivative of the function $f(t, x)$; see \eqref{faz}. To this end we begin with formula
\eqref{ujh}. We have
\begin{eqnarray}\nonumber\frac{d^2 f}{dt^2}&=& \frac{d }{dt}\left[\frac{d\,|\P\ct\nu^t|}{dt}\,\mathcal{J}ac\,\vartheta_t + |\P\ct\nu^t|\frac{d\,\mathcal{J}ac\,\vartheta_t}{dt}\right] \\\nonumber&=&  \frac{d^2\,|\P\ct\nu^t|}{dt^2}\,\mathcal{J}ac\,\vartheta_t +2\frac{d\,|\P\ct\nu^t|}{dt}\frac{d\,\mathcal{J}ac\,\vartheta_t}{dt}+  |\P\ct\nu^t|\frac{d^2\,\mathcal{J}ac\,\vartheta_t}{dt^2}.
\end{eqnarray}At a first glance, it is clear that only the first term is not bounded near the characteristic set $C_\UU$. 
More precisely, it is elementary to see that $$\frac{d^2\,|\P\ct\nu^t|}{dt^2}=\frac{\left|\frac{d\,\P\ct\nu^t}{dt}\right|^2-\left\langle\frac{d\,\P\ct\nu^t}{dt},\nn^t \right\rangle^2}{|\P\ct\nu^t|}+\left\langle\frac{d^2\,\P\ct\nu^t}{dt^2},\nn^t \right\rangle.$$This shows that, in order to differentiate under the integral sign, we  need the following further hypothesis:
\begin{itemize}
 \item[{$\mathbf (A_3)$}]  \it there exists $h\in L^1\left( \UU; \sigma\rr^{n-1}\right)$ such that $\frac{1}{|\P\ct\nu^t|}\leq h$ for every $t\in]-\epsilon, \epsilon[$.
\end{itemize}

\begin{oss}\label{remremrem}Using $(A_3)$ it is not difficult to show the validity of $(A_2)$ in Warning \ref{lwar}. 
\end{oss}

We  continue our proof of the 2nd variation of $\per$ with the calculation of
$\ddot{\Gamma}(t)$ at a fixed non-characteristic point $p_0\in\UU\setminus C_\UU$. To this end, we start from the following formula:
\begin{eqnarray}\label{IIoO}\ddot{\Gamma}(t)=\vartheta_t^\ast\left(\mathcal{L}_{\WW}\left( \WW\LL d  \pert\right)+
\, \mathcal{L}_{\WW}\,d \left( \WW\LL\pert\right)\right).\end{eqnarray}
As already said, the 2nd time-derivative of $\Gamma(t)$ can still be computed  as a Lie derivative. Moreover, since $d\circ\Lie=\Lie\circ d$, we have
 
\begin{eqnarray}\label{2vi} \ddot{\Gamma}(t)&=&\vartheta_t^\ast\left(\underbrace{\mathcal{L}_{\WW}\left( \WW\LL d  \pert\right)}_{=:A}+
 d  \,\underbrace{\mathcal{L}_{\WW}\left( \WW\LL\pert\right)}_{=:B}\right).\end{eqnarray}
The calculation of $A=\mathcal{L}_{\WW}\left( \WW\LL d  \pert\right)$ is the \textquotedblleft hard\textquotedblright\, part of the 2nd variation formula and will be done in the sequel. So let us preliminarily  consider the quantity $B=\mathcal{L}_{\WW}\left( \WW\LL\pert\right)$. (In the next  calculations we will use the following general identity for Lie derivatives of a differential form $\omega$: \begin{equation}\label{fpa}\Lie_Z(Y\LL\omega)=[Z,Y]\LL\omega+Y\LL\Lie_Z\omega;
\end{equation}see \cite{Spiv}, Ch. 9, p. 515). We have

\begin{eqnarray*}B&=&\Lie_{\WW}\left( \WW\LL\pert\right)\\&=&
 \Lie_{\WW}\left( \underbrace{\left( \WW\ot|\P\ct\nu^t|-\langle \WW,\nu^t\rangle{\nt}\ot\right)}_{=:\widetilde{Y}} \LL(\sigma^{n-1}\rr)_t\right)\qquad\mbox{(by Lemma \ref{fondam})}\\&=& [\WW, \widetilde{Y}]\LL(\sigma^{n-1}\rr)_t + \widetilde{Y}\LL\Lie_{\WW}(\sigma^{n-1}\rr)_t\qquad\mbox{(by  \ref{fpa})}\\&=& [\WW, \widetilde{Y}]\ot\LL(\sigma^{n-1}\rr)_t + \widetilde{Y}\LL\underbrace{\left(-\langle\WW,\nu^t\rangle(\mathcal{H}\rr)_t +\div\tut \left(\WW\ot\right)\right)}_{=:g_t} (\sigma^{n-1}\rr)_t\quad\mbox{(by the 1st variation  of $(\sigma^{n-1}\rr)_t$)}
\\&=&\left( [\WW, \widetilde{Y}]\ot+g_t\widetilde{Y}\right)\LL (\sigma^{n-1}\rr)_t.
\end{eqnarray*}
Therefore, the second term in formula \eqref{2vi}, i.e. $d B$, is given  by

\begin{equation}\label{dB}dB=d\left\lbrace \left( [\WW, \widetilde{Y}]\ot+g_t\widetilde{Y}\right)\LL (\sigma^{n-1}\rr)_t\right\rbrace=\div\tut\left(  [\WW, \widetilde{Y}]\ot+g_t\widetilde{Y}\right) (\sigma^{n-1}\rr)_t.
\end{equation}

\noindent{\bf Step 0.}\, {[Divergence-type terms]}.\,\, Set $t=0$.  First, note that $[\WW, \widetilde{Y}]\ot\big|_{t=0}$ is a vector field of class $\cont^1$ out of $C_\UU$. We also stress that 
$$[\WW, \widetilde{Y}]\ot\big|_{t=0}=[\WW, \WW\ot]\ot\big|_{t=0}-W(\langle W,\nu\rangle)\nn\ot-\langle W,\nu\rangle[\WW, {\nn^t}\ot]\ot\big|_{t=0}.$$ Clearly, the first term is a vector field of class  $\cont^1$. 
The second term is the product of a $\cont^1$ function times the vector field  $\nn\ot$. Although not defined at $C_\UU$,  this term  belongs to $L^\infty$. Furthermore $$[\WW, {\nn^t}\ot]\ot\big|_{t=0}=\left[\WW,\left({\nn^t}-|\P\ct\nu^t|\nu^t\right)\right]\ot\Big|_{t=0}=[\WW,\nn^t]\ot\big|_{t=0}-|\P\cc\nu|\,[\WW,\nu^t]\big|_{t=0}.$$By using the very definition of $\nn^t=\frac{\P\ct\nu^t}{|\P\ct\nu^t|}$, we easily see that $[\WW,\nn^t]\ot\big|_{t=0}$ can be estimated near the characteristic set $C_\UU$ by (a constant times) the function $\frac{1}{|\P\cc\nu|}$. Continuing this argument, it is not difficult to
realize that the tangential divergence of  $ [\WW, \widetilde{Y}]\ot$, at $t=0$, i.e. $\div\ts[\WW, \widetilde{Y}]\ot\big|_{t=0}$ , can be estimated, locally around $C_\UU$, by  (a constant times) the function $\frac{1}{|\P\cc\nu|^2}$. An analogous  argument can be repeated
for the second divergence-type term in formula \eqref{dB}. In fact, since the function $g_t$ is of class $\cont^1$  on $\UU_t$ for all $t\in]-\epsilon, \epsilon[$, we  see that, near the characteristic set $C_\UU$, the function $\div\tu(g_0 Y)$ can be  estimated by (a constant times) the function $\frac{1}{|\P\cc\nu|}$.
 
\begin{oss}The previous estimates show that, in order to integrate the divergence-type term $dB$ on $\UU$, we need a further condition. More precisely, we have  (at least) to require that  $\frac{1}{|\P\cc\nu|^2}\in L^1(\UU, \sigma\rr^{n-1})$.
\end{oss}

\noindent{\bf Step 1.}\, We start with the calculation of the term $A$ in formula \eqref{2vi}. 

\begin{war}\label{lwar2}
In order to simplify our calculations, hereafter we shall assume  \rm $\MS$ to be constant.
\end{war}

\begin{oss} \label{aria}We stress that if  $\MS=\mbox{\rm{const.}}$, then
$\mathcal{L}_X\MS=0$ for all $X\in\XX^1(\TS)$. In particular, if $W$ denotes
the variation vector of $\vartheta_t$, we have
$\iu^\ast\left(\mathcal{L}_{\WW\ss}\MST\right)=\mathcal{L}_{W\ss}\MS=0.$ 
Analogously, we have
$\iu^\ast\left(\mathcal{L}_{\TB_\alpha}\MST\right)
=\mathcal{L}_{\TB_\alpha}\MS=0$ for all $\alpha\in I\vv$. Hence
\begin{equation}\label{dertang}\iu^\ast\left(\mathcal{L}_{\tau_{\alpha}}\MST\right)
=\iu^\ast\left(\mathcal{L}_{\varpi^t_{\alpha}\nt}\MST
\right)\qquad\forall\,\,\alpha\in
I\vv.\end{equation}
\end{oss}

If $t=0$, we have
\begin{eqnarray*} A|_{t=0}&=& \iu^\ast\left(\mathcal{L}_{\WW}\left(-\cn_t\,\MST
\pert\right)\right)\\&=& \left(-\cn\MS\Lie_{\WW}\pert - W(\cn)
\MS-\cn\, \iu^\ast\left(\mathcal{L}_{\WW} \MST\right)\right)\,\per
\\&=& -\cn\MS\left(-\MS\langle W,\nu\rangle +\div\tu\left( W\ot|\PH\nu|-\langle W,\nu\rangle\nn\ot \right)\right)\sigma\rr^{n-1}  -\left(  W(\cn) \MS +\cn\,
\iu^\ast\left(\mathcal{L}_{\WW} \MST\right)\right)\,\per\\&=&\left( \MS^2\cn^2- W(\cn) \MS -\iu^\ast\left(\mathcal{L}_{\WW} \MST\right)\right)\per -\cn\MS\div\tu\left( W\ot|\PH\nu|-\langle W,\nu\rangle\nn\ot \right)\sigma\rr^{n-1},\end{eqnarray*}where we have used the 1st variation of $\per$.\\

\noindent{\bf Step 2.}\, 
Setting
$W\spp:=w_1\nn+ W\vv$, where $W\vv=\sum_{\alpha\in I\vv}w_\alpha\tau_\alpha$, we get that
\begin{eqnarray}\label{nevenam}\iu^\ast\left(\mathcal{L}_{\WW}\MST\right)
&=&\iu^\ast\left(\mathcal{L}_{\WW\ss}\MST\right) +
\iu^\ast\left(\mathcal{L}_{\WW\spp}\MST\right)\\\nonumber  &=&
\iu^\ast\left(\mathcal{L}_{\WW\spp}\MST\right)\qquad\mbox{(by Remark
\ref{aria})}\\\nonumber
&=&\iu^\ast\left(\mathcal{L}_{\ww_1\nt}\MST\right) + \sum_{\alpha\in
I\vv}\iu^\ast\left(\mathcal{L}_{\ww_\alpha\tau_\alpha}\MST\right)\\\nonumber
&=&\iu^\ast\left(\mathcal{L}_{\ww_1\nt}\MST\right) + \sum_{\alpha\in
I\vv} \iu^\ast\left(\mathcal{L}_{
\ww_\alpha\varpi^t_\alpha\nt}\MST\right)\qquad\mbox{(by
\eqref{dertang})}\\&=&\nonumber
\iu^\ast\left(\mathcal{L}_{\cn_t\,\nt}\MST\right).
\end{eqnarray}

\noindent{\bf Step 3.}\, From  Step 2, we see that it remains to calculate
$\mathcal{L}_{\cn_t\,\nt}\MST=\cn_t\frac{\partial
\MST}{\partial\nt}.$  This will be done by using an adapted frame
$\underline{\tau}=\{\tau_1,...,\tau_n\}$ to $\UU$ which satisfies  Lemma \ref{sple} at  $p_0\in\UU$. We also recall that $\tau_1(x)=\nn(x)$ for every $x\in\UU$. We compute \begin{eqnarray*} -\dfrac{\partial
\MST}{\partial\nt}&=&\sum_{j\in
I\ss}\dfrac{\partial}{\partial\tau_1}\left\langle\nabla_{\tau_j}\tau_1,\tau_j\right\rangle\\&=&\sum_{j\in
I\ss}\left(\left\langle\nabla_{\tau_1}\nabla_{\tau_j}\tau_1,\tau_j\right\rangle+\left\langle\nabla_{\tau_j}\tau_1,\nabla_{\tau_1}\tau_j\right\rangle\right)\\&=&\sum_{j\in
I\ss}\left(\left\langle\left(\nabla_{\tau_1}\nabla_{\tau_j}\tau_1\mp
\nabla_{\tau_j}\nabla_{\tau_1} \tau_1\mp \nabla_{[\tau_1,
\tau_j]}\tau_1\right),\tau_j\right\rangle
+\sum_{k=2}^n\left\langle\nabla_{\tau_j}\tau_1,
\tau_k\right\rangle\left\langle\nabla_{\tau_1}\tau_j,\tau_k\right\rangle\right)\\&=&\sum_{j\in
I\ss}\left(
-\Phi_{1j}(\tau_1,\tau_j)+\left\langle\nabla_{\tau_j}\nabla_{\tau_1}
\tau_1,\tau_j\right\rangle+\left\langle\nabla_{[\tau_1,
\tau_j]}\tau_1,\tau_j\right\rangle +\sum_{\alpha\in
I\vv}^n\left\langle\nabla_{\tau_j}\tau_1, \tau_\alpha
\right\rangle\left\langle\nabla_{\tau_1}\tau_j,\tau_\alpha\right\rangle\right)
\end{eqnarray*}where we have used the definition of $\Phi_{1j}(\tau_1,\tau_j)$ and the fact (Lemma \ref{sple}) that $\phi_{jk}=0$ at $p_0\in\UU$ for every $j, k\in I\ss$. We have
\begin{eqnarray*} \left\langle\nabla_{[\tau_1, \tau_j]}\tau_1,\tau_j\right\rangle&=&C_{1j}^1 \phi_{1j}(\tau_1)+\sum_{k\in I\ss} C_{1j}^k\phi_{1j}(\tau_k)+\sum_{ \alpha \in I\vv}C_{1j}^\alpha\phi_{1j}(\tau_\alpha)\\&=&-\left(\phi_{1j}(\tau_1)\right)^2-\sum_{k\in I\ss} \phi_{1k}(\tau_j)\,\phi_{1j}(\tau_k)-\sum_{ \alpha \in I\vv}\left\langle C\cc^\alpha\tau_1,\tau_j\right\rangle\phi_{1j}(\tau_\alpha).
\end{eqnarray*}Moreover, using (vii) of Lemma \ref{sple} yields
\begin{eqnarray*} \left\langle\nabla_{\tau_j}\tau_1, \tau_\alpha \right\rangle\left\langle\nabla_{\tau_1}\tau_j,\tau_\alpha\right\rangle=\phi_{1\alpha}(\tau_j)\phi_{j\alpha}(\tau_1)=-\dfrac{1}{4}\left\langle C^\alpha\cc\tau_1,\tau_j
\right\rangle^2.
\end{eqnarray*}
 Therefore, Lemma \ref{riccri} implies that
\begin{eqnarray*} -\dfrac{\partial \MST}{\partial\nt}&=&  \frac{1}{2}\sum_{\alpha\in I\cd}
|C\cc^\alpha\tau_1|^2+\div\sst\left(\nabla_{\tau_1}\tau_1\right)
 - \sum_{j, k\in I\ss\,\alpha\in I\vv}  \left(\left(\phi_{1j}(\tau_1)\right)^2+  \phi_{1k}(\tau_j)\phi_{1j}(\tau_k)+ \left\langle C\cc^\alpha\tau_1,\tau_j\right\rangle\phi_{1j}(\tau_\alpha)
 \right).
\end{eqnarray*}Hence, from Lemma \ref{kl}, Lemma \ref{f} and formula \eqref{b} we get that
\begin{eqnarray*} -\dfrac{\partial \MST}{\partial\nt}&=& \frac{1}{2}\sum_{\alpha\in I\cd}
|C\cc^\alpha\tau_1|^2- \div\sst\left(\frac{\grad\sst\cn_t}{\cn_t}+ C\cc(\varpi^t\cd)\tau_1\right)
\\&&- \left|\frac{\grad\sst\cn_t}{\cn_t}+ C\cc(\varpi^t\cd)\tau_1\right|^2+\|A^t\cc\|\ngr^2-\|S^t\cc\|\ngr^2\\&&- \sum_{j\in I\ss \alpha \in I\vv}\left\langle C\cc^\alpha\tau_1,\tau_j\right\rangle\left(\frac{1}{2}\left\langle
C^\alpha\cc\tau_1,\tau_j\right\rangle-  \left\langle C(\varpi^t)\tau_\alpha,\tau_j\right\rangle+  {\tau_j( \varpi^t_\alpha)}- \varpi^t_\alpha \frac{\tau_j(\cn_t)}{\cn_t}\right).
\\&=&  - \div\sst\left(\frac{\grad\sst\cn_t}{\cn_t}+ C\cc(\varpi^t\cd)\tau_1\right)
 - \left|\frac{\grad\sst\cn_t}{\cn_t}+ C\cc(\varpi^t\cd)\tau_1\right|^2+\|A^t\cc\|\ngr^2-\|S^t\cc\|\ngr^2\\&&+ \sum_{j\in I\ss \alpha \in I\vv}\left\langle C\cc^\alpha\tau_1,\tau_j\right\rangle\left( \left\langle C(\varpi^t)\tau_\alpha,\tau_j\right\rangle- {\tau_j( \varpi^t_\alpha)}+ \varpi^t_\alpha \frac{\tau_j(\cn_t)}{\cn_t}\right)\\&=&  - \div\sst\left(\frac{\grad\sst\cn_t}{\cn_t}+ C\cc(\varpi^t\cd)\tau_1\right)
 - \left|\frac{\grad\sst \cn_t}{\cn_t}+ C\cc(\varpi^t\cd)\tau_1\right|^2+\|A^t\cc\|\ngr^2-\|S^t\cc\|\ngr^2\\&&+\sum_{\alpha \in I\vv}\left(\left\langle C\cc^\alpha\tau_1,C(\varpi^t)\tau_\alpha\right\rangle-\left\langle C\cc^\alpha\tau_1, \grad\sst\varpi^t_\alpha\right\rangle\right) + \left\langle C\cc(\varpi\cd^t)\tau_1, \frac{\grad\sst \cn_t}{\cn_t}\right\rangle\\&=& -\dfrac{\Delta\sst\cn_t}{\cn_t} - \div\sst\left( C\cc(\varpi^t\cd)\tau_1\right)
 - \left|C\cc(\varpi^t\cd)\tau_1\right|^2-2\left\langle\dfrac{\grad\sst\cn_t}{\cn_t} , C\cc(\varpi^t\cd)\tau_1\right\rangle+\|A^t\cc\|\ngr^2-\|S^t\cc\|\ngr^2\\&&+\sum_{\alpha \in I\vv}\left(\left\langle C\cc^\alpha\tau_1,C(\varpi^t)\tau_\alpha\right\rangle- \left\langle C\cc^\alpha\tau_1, \grad\sst\varpi^t_\alpha\right\rangle\right) + \left\langle C\cc(\varpi\cd^t)\tau_1, \frac{\grad\sst \cn_t}{\cn_t}\right\rangle\\&=&-\dfrac{\Delta\sst\cn_t}{\cn_t} - \div\sst\left( C\cc(\varpi^t\cd)\tau_1\right)
 - \left|C\cc(\varpi^t\cd)\tau_1\right|^2-\left\langle\dfrac{\grad\sst\cn_t}{\cn_t} , C\cc(\varpi^t\cd)\tau_1\right\rangle+\|A^t\cc\|\ngr^2-\|S^t\cc\|\ngr^2\\&&+\sum_{\alpha \in I\vv}\left(\left\langle C\cc^\alpha\tau_1,C(\varpi^t)\tau_\alpha\right\rangle-\left\langle C\cc^\alpha\tau_1, \grad\sst\varpi^t_\alpha\right\rangle\right) \\&=&-\dfrac{\mathcal{L}\sst\cn_t}{\cn_t} -  \mathcal{D}\sst\left( C\cc(\varpi^t\cd)\tau_1\right)
 +\|A^t\cc\|\ngr^2-\|S^t\cc\|\ngr^2 +\sum_{\alpha \in I\vv}\left(\left\langle C\cc^\alpha\tau_1,C(\varpi^t)\tau_\alpha\right\rangle- \left\langle C\cc^\alpha\tau_1, \grad\sst\varpi^t_\alpha\right\rangle\right).
\end{eqnarray*}

 Under the assumptions made in  Warning \ref{lwar}, Remark  \ref{remremrem} and Warning \ref{lwar2}, we can now achieve the proof of the 2nd variation  of $\per$.\\

\noindent{\bf Step 4.}\, We know by  Remark \ref{remdiff} that if $\frac{1}{|\PH\nu|^2}\in L^2(\UU, \per)$, then  the function $\cn^2$  is admissible; see Definition \ref{adm}. Hence, if $\frac{1}{|\P\cc\nu|^2}\in L^2\left(S;\per\right)$,  then
\begin{eqnarray}\nonumber II_\UU(W,\per)&=& \int_{\UU}\left((\cn\,\MS)^2-
W(\cn) \MS -\cn\,
\iu^\ast\left(\mathcal{L}_{\cn_t\,\nt}\MST\right)\right)\,\per\\\nonumber&& + \int_\UU\left(\div\tu \left( [\WW, \widetilde{Y}]\ot\big|_{t=0} +g_0 Y\right)- \cn\MS\div\tu\left( W\ot|\PH\nu|-\langle W,\nu\rangle\nn\ot \right)\right) \sigma\rr^{n-1}\\\nonumber&=& \int_{\UU}\Bigg\{ -
W(\cn) \MS +\cn^2\left((\MS)^2+\|A\cc\|\ngr^2-\|S\cc\|\ngr^2  \right)-\cn\,{\lll\cn}
  \\\nonumber&& +\cn^2\bigg[- \mathcal{D}\ss \left( C\cc(\varpi\cd)\tau_1\right)+ \sum_{\alpha \in I\vv}\left(\left\langle C\cc^\alpha\tau_1,C(\varpi)\tau_\alpha\right\rangle-\left\langle C\cc^\alpha\tau_1, \grad\ss\varpi_\alpha\right\rangle\right)\bigg]\Bigg\}\per\\\nonumber&& + \int_\UU\left(\div\tu \left( [\WW, \widetilde{Y}]\ot\big|_{t=0}+g_0 Y \right)- \cn\MS\div\tu\left( W\ot|\PH\nu|-\langle W,\nu\rangle\nn\ot \right)\right) \sigma\rr^{n-1}\end{eqnarray}
\begin{eqnarray}\nonumber&=&
 \int_{\UU}\Bigg \{-
W(\cn) \MS +\cn^2\left((\MS)^2+\|A\cc\|\ngr^2-\|S\cc\|\ngr^2  \right)+|\qq\cn|^2\qquad\mbox{(by formula \eqref{eq})}
  \\&& \label{iuiul} +\cn^2\sum_{\alpha \in I\vv} \left(-\varpi_\alpha\lh\left(C\cc^\alpha\tau_1\right)+\left\langle C\cc^\alpha\tau_1,C(\varpi)\tau_\alpha\right\rangle-2\left\langle C\cc^\alpha\tau_1, \grad\ss\varpi_\alpha\right\rangle\right)\Bigg\}\,\per\\&&\nonumber + \int_\UU\left(\div\tu \left( [\WW, \widetilde{Y}]\ot\big|_{t=0} +g_0 Y\right)- \cn\MS\div\tu\left( W\ot|\PH\nu|-\langle W,\nu\rangle\nn\ot \right)\right) \sigma\rr^{n-1}.\end{eqnarray} Using  Lemma \ref{mk} yields
\begin{eqnarray*} 
&& \int_{\UU}\Bigg\{ -
W(\cn) \MS +\cn^2\left((\MS)^2-\|A\cc\|\ngr^2-\|S\cc\|\ngr^2  \right)+|\qq\cn|^2
  \\&&  +\cn^2\sum_{\alpha \in I\vv} \bigg[- |C\cc(\varpi\cd)\tau_1|^2+\left\langle C\cc^\alpha\tau_1,C(\varpi)\tau_\alpha\right\rangle-2\left\langle C\cc^\alpha\tau_1, \grad\ss\varpi_\alpha\right\rangle \bigg]\Bigg\}\,\per\\&=&
 \int_{\UU}\bigg\{ -
W(\cn) \MS +\cn^2\left((\MS)^2-\|A\cc\|\ngr^2-\|S\cc\|\ngr^2  \right)+|\qq\cn|^2
  \\&&  -\cn^2\sum_{\alpha\in I\vv} \left\langle
\left(2\,\qq(\varpi_\alpha)-C(\varpi)\TB_\alpha\right),C^\alpha
\tau_1\right\rangle\bigg\}\,\per,\end{eqnarray*}where we recall that $\TB_\alpha=\tau_\alpha-\varpi_\alpha\tau_1$ and that $\tau_1=\nn$, $\tau_\alpha=X_\alpha$ for any $\alpha\in I\vv$; see Definition \ref{movadafr}. Finally, using the last identity in \eqref{iuiul} yields the following:
\begin{teo}[2nd variation of $\per$]\label{2vg}Let $S\subset\GG$ be a compact  
hypersurface  of class $\cont^3$ with, or without, boundary and let
 $\vartheta: ]-\epsilon,\epsilon[\times S
\rightarrow \GG$ be a ${\cont}^3$-smooth variation of $S$. Let $W=\frac{d\,\vartheta_t}{dt}\big|_{t=0}$ be the variation vector field, let $\cn:=\frac{\langle W\op,\nu\rangle}{|\P\cc\nu|}$ and let $W\op,\,W\ot$ be the normal and  tangential components of $W$ along $S$, respectively. 
We further assume that: \begin{itemize}
 \item [$\rm (i)$] there exists\footnote{Alternatively, we can assume the validity of $(A_2)$ in Warning \ref{lwar}.} $h\in L^1\left( \UU; \sigma\rr^{n-1}\right)$ such that $\frac{1}{|\P\ct\nu^t|}\leq h$ for every $t\in]-\epsilon, \epsilon[$;
\item [$\rm (ii)$]the horizontal mean curvature $\MS$ of $S$ is constant;\item [$\rm (iii)$] the function $\frac{1}{|\P\cc\nu|^2}\in L^2\left(S;\per\right)$.
\end{itemize}Then 
 \begin{eqnarray}\nonumber
 II_S(W,\per)&=& \int_{S}\bigg\{ -
W(\cn) \MS +\cn^2\left((\MS)^2-\|A\cc\|\ngr^2-\|S\cc\|\ngr^2  \right)+|\qq\cn|^2
  \\\label{1linea}&&  -\cn^2\sum_{\alpha\in I\vv} \left\langle
\left(2\,\qq(\varpi_\alpha)-C(\varpi)\TB_\alpha\right),C^\alpha
\nn\right\rangle\bigg\}\,\per\\\nonumber&&+  \int_S\left\lbrace \div\ts \left( [\WW, \widetilde{Y}]\ot\big|_{t=0} +g_0 Y\right)- \cn\MS\div\ts\left( W\ot|\PH\nu|-\langle W,\nu\rangle\nn\ot \right) \right\rbrace \sigma\rr^{n-1}\end{eqnarray}where $\widetilde{Y}:=\WW\ot|\P\ct\nu^t|-\langle \WW,\nu^t\rangle{\nt}\ot$,\, $Y=\widetilde{Y}|_{t=0}$ and \,$g_0=\left(-\langle W\op,\nu \rangle \mathcal{H}\rr +\div\ts W\ot\right)$.
\end{teo}\begin{proof}If $C_S\neq\emptyset$, then $\rm (i)$ implies the possibility to differentiate  under the integral sign the function $f(t,x)$ defined by formula \eqref{faz}. This has been done by using the machinery of differential forms. This way we have obtained \eqref{1linea} by further assuming that $\MS$ is constant. Nevertheless,  we have to take care of the existence of the  involved integrals. The integrability of the divergence-type terms has been already discussed at Step 0. We recall that if $\frac{1}{|\P\cc\nu|^2}\in L^1(S; \sigma\rr^{n-1})$ it follows that all these terms are integrable. Clearly, the latter condition is automatically implied by  (iii).   Moreover, the condition $\frac{1}{|\PH\nu|^2}\in L^2(S, \per)$ implies that function $\cn^2$ is admissible; see Definition \ref{adm}. Hence,  using formula  \eqref{eq},  we see that the function $-\cn\,\lll\cn$ can be integrated by parts, as previously done. Furthermore, a rather tedious (but completely elementary) analysis shows that the same condition implies that each term in \eqref{1linea} is integrable over $S$.  More precisely, the integral of each of these terms can be estimated, near the characteristic set $C_S$, by (a constant times) $\int_S\frac{1}{|\P\cc\nu|^4}\,\per$. (Note that these estimates follow basically from  the same calculation performed in formula \eqref{gghh} of Remark \ref{remfac}. In particular, one uses the following $$X|\P\cc\nu|=\frac{\left\langle\left[ \mathcal{J}\rr(\P\cc\nu)\right]^{Tr}\,\P\cc\nu, X\right\rangle}{|\P\cc\nu|}$$ for every $X\in\XG$. An analogous argument was made at Step 0). This achieves the proof.
 \end{proof}

\begin{corollario}\label{corvar2stat}Let the assumptions of Theorem \ref{2vg} hold and let $\vartheta$ be compactly supported on $S$.  
If $S$ is $\HH$-minimal, i.e.  $\MS=0$, then
 \[II_S(W,\per)=\int_{S}\left\{ |\qq\cn|^2   -\cn^2\left(\|A\cc\|\ngr^2+\|S\cc\|\ngr^2 +\sum_{\alpha\in I\vv} \left\langle
\left(2\,\qq(\varpi_\alpha)-C(\varpi)\TB_\alpha\right),C^\alpha
\nn\right\rangle\right) \right\}\,\per.\]\end{corollario}
\begin{proof}We just have to analyze the 2nd integral in formula \eqref{1linea}. We already know that $Y$ is admissible; see Corollary \ref{corvar1ma}. Since $g_0$ is $\cont^1$-smooth on $S$ and $g_0=0$ on $\partial S$, we can conclude that $g_0Y$ is admissible. Hence $\int_S \div\ts\left(g_0Y\right)\,\sigma\rr^{n-1}=\int_{\partial S}\langle g_0Y, \eta\rangle\,\sigma\rr^{n-2}=0$. Furthermore, since $\MS=0$ and $[\WW, \widetilde{Y}]\ot \big|_{t=0}=0$ on $\partial S$, we just have to show that $[\WW, \widetilde{Y}]\ot \big|_{t=0}$ is admissible. More precisely, below we shall prove that $[\WW, \widetilde{Y}]\ot \big|_{t=0}$ satisfies the assumptions in Proposition \ref{divtoeRiemW12}. Under our assumptions, this can  be seen as follows. First, note that  $\WW$ is of class $\cont^2$ on $]-\epsilon, \epsilon[\times S$ and that $\widetilde{Y}=\WW\ot|\P\ct\nu^t|-\langle \WW,\nu^t\rangle{\nt}\ot$ is of class $\cont^1$ on $]-\epsilon, \epsilon[\times(S\setminus C_S)$. Moreover, we have
\begin{eqnarray}\nonumber[\WW, \widetilde{Y}]\ot \big|_{t=0}&=&\left.\left( \nabla_{\WW}\widetilde{Y}-\nabla_{\widetilde{Y}}\WW\right) \right|_{t=0}\\&=&\left.\left[|\P\ct\nu^t|\nabla_{\WW}\WW\ot+ \WW(|\P\ct\nu^t|)\WW\ot-\left(\WW(\langle \WW,\nu^t\rangle){\nt}\ot+ \langle \WW,\nu^t\rangle\nabla_{\widetilde{Y}}{\nt}\ot\right) \right]\right|_{t=0}.\end{eqnarray}
\it We claim that  $[\WW, \widetilde{Y}]\ot \big|_{t=0}\in W_{comp}^{1, 2}(S; \TS)$. \rm
In fact, the first addend is Lipschitz, the second and third addends are in $L^\infty$, and the fourth addend can be estimated by (a constant times) $\frac{1}{|\PH\nu|}$. It is worth remarking that the estimate of the fourth addend relies on the fact that $$\mathcal{J}\rr\nn=\mathcal{J}\rr\left(\frac{\P\cc\nu}{|\P\cc\nu|}\right)= \frac{\mathcal{J}\rr\P\cc\nu-\nn\otimes \grad\rr|\P\cc\nu|}{|\P\cc\nu|}.$$In particular, under our assumptions, we have $\frac{1}{|\PH\nu|}\in L^2(S; \sigma\rr^{n-1})$. Hence,  $[\WW, \widetilde{Y}]\ot \big|_{t=0}\in L^2(S; \sigma\rr^{n-1})$.
Continuing this argument, we easily see that each tangential derivative of 
$[\WW, \widetilde{Y}]\ot \big|_{t=0}$ can be estimated by (a constant times)  $\frac{1}{|\PH\nu|^2}$ and the claim follows since $\frac{1}{|\PH\nu|^2}\in L^2\left(S; \sigma\rr^{n-1}\right)$. \end{proof}
 
\begin{no}\label{415}For the sake of simplicity, we shall set:
\begin{eqnarray}\label{tranx}\mathcal{B}\ts:=\underbrace{\|S\cc\|\ngr^2 +\|A\cc\|\ngr^2}_{=\|B\cc\|^2\ngr} +\sum_{\alpha\in I\vv} \left\langle
 \left(2 \qq(\varpi_\alpha)-C(\varpi)\TB_\alpha\right),C^\alpha
\tau_1\right\rangle.\end{eqnarray}\end{no}

We stress that, unlike the Euclidean case where
$\mathcal{B}\ts:=\|B\rr\|\ngr^2$, it is not necessarily true that
$\mathcal{B}\ts\geq 0$;  an example of this fact can be found in
Section \ref{exf2}; see Remark \ref{urme}.
\begin{oss}[Heisenberg group; see Example \ref{hng2}]Let $S\subset\mathbb H^n$ be $\HH$-minimal and set $\nn^\circ:=-C\cc^{2n+1}\nn$. Then, we have\begin{equation}\label{sadf}
\mathcal{B}\ts=\|S\cc\|^2\ngr-\left(
2\frac{\partial\varpi}{\partial\nn^\circ} -\frac{n+1}{2}\varpi^2
\right).
\end{equation}
\end{oss}

\section{Geometric identities for constant $\HH$-mean curvature hypersurfaces}\label{PCARITGC}

\begin{lemma}\label{ljjjkl} Let $S\subset\GG$
be a  hypersurface  of class $\cont^2$ and let $\phi\in\mathbf{C}^2(\GG)$. Then
we have $$\Delta\ss\phi=\Delta\cc\phi +
\MS\frac{\partial\phi}{\partial\nn}-\left\langle{\rm
Hess}\cc\phi\,\nn,\nn\right\rangle$$at each non-characteristic point $p\in S\setminus C_S$.
\end{lemma}
\begin{proof}First, note that we can use  the invariant definition of the Laplacian on vector bundles; see, for instance, \cite{Ch1}. So we have
 \begin{eqnarray*}\Delta\cc\phi&=&\sum_{i\in I\cc}\left(\tau^{(2)}_i-\gc_{\tau_i}\tau_i\right)(\phi)
 \\&=&\tau^{(2)}_1
 (\phi)-\left(\gc_{\tau_1}\tau_1\right)(\phi)+\sum_{i\in I\ss}\left( \left(\tau^{(2)}_i-
 \gs_{\tau_i}\tau_i\right)(\phi)-\left\langle\gc_{\tau_i}\tau_i,
 \nn\right\rangle\frac{\partial\phi}{\partial\nn}\right)\\&=&\tau^{(2)}_1
 (\phi)-\left(\gc_{\tau_1}\tau_1\right)(\phi)+ \Delta\ss\phi
 - \MS\frac{\partial\phi}{\partial\nn}.
 \end{eqnarray*}Now we claim that $\tau^{(2)}_1
 (\phi)-\left(\gc_{\tau_1}\tau_1\right)(\phi)=\left\langle{\rm
Hess}\cc(\phi)\nn,\nn\right\rangle.$ To prove this claim, set $\tau_1=\sum_{i\in
I\cc}A^1_iX_i$ and compute
\begin{eqnarray*}\tau^{(2)}_1(\phi)=\sum_{i\in
I\cc}\left(\tau_1(A^1_iX_i(\phi))\right)=\sum_{i,j\in
I\cc}\left(\tau_1(A^1_i)X_i(\phi) +
A^1_iA^1_jX_j(X_i(\phi))\right).\end{eqnarray*}Since
$\gc_{\tau_1}\tau_1=\sum_{i, j\in I\cc}\left(\tau_1(A^1_i)X_i+
A^1_iA^1_j\underbrace{\gc_{X_i}X_j}_{=0}\right)$, we  get
that$$\tau^{(2)}_1
 (\phi)-\left(\gc_{\tau_1}\tau_1\right)(\phi)=\sum_{i,
j\in I\cc} A^1_iA^1_jX_j(X_i(\phi))=\left\langle{\rm
Hess}\cc(\phi)\nn,\nn\right\rangle,$$as wished.
 \end{proof}

\begin{lemma}\label{Doca}Let $S\subset\GG$ be a  $\cont^2$  non-characteristic hypersurface of constant horizontal mean curvature $\MS$. Then, the following identities hold:
\begin{itemize}\item[{\rm(i)}] $\sum_{i\in I\ss}\left\langle\gc_{\tau_i}\gc_{\tau_i}\nn,\nn\right\rangle=-\|B\cc\|^2\ngr$;
\item[ {\rm(ii)}]$\sum_{i\in
I\ss}\left\langle\gc_{\tau_i}\gc_{\tau_i}\nn,\tau_k\right\rangle =-\Big(\left\langle\gc_{\nn}\nn,C\ss(\varpi\cd)\tau_k\right\rangle + \sum_{\alpha\in
I\vv}\left\langle C^\alpha\cc\qq\varpi_\alpha,\tau_k\right\rangle +\MS\left\langle
C\cc(\varpi\cd)\nn,\tau_k\right\rangle -B\cc(C\cc(\varpi\cd)\nn,\tau_k)\Big)\quad \forall\,\,k\in I\ss.$ 
\end{itemize}
\end{lemma}

\begin{proof}Throughout this proof, we  use  an adapted frame as in Lemma \ref{sple}. Fix a point $p\in S$.

\noindent {\it Proof of (i)}. Since
 $\left\langle\nn,\nn\right\rangle=1$ we get that $\left\langle\gc_{\tau_i}\nn,\nn\right\rangle=0\,\,\forall\,\,i\in I\ss$. So, we have
\begin{eqnarray*}\sum_{i\in
I\ss}\left\langle\gc_{\tau_i}\gc_{\tau_i}\nn,\nn\right\rangle&=&- \sum_{i\in
I\ss}\left\langle\gc_{\tau_i}\nn,\gc_{\tau_i}\nn\right\rangle\\=-
\sum_{i,j,k\in
I\ss}\left\langle\gc_{\tau_i}\nn,{\tau_j}\right\rangle\left\langle\gc_{\tau_i}\nn,{\tau_k}\right\rangle\left\langle\tau_j,\tau_k\right\rangle&=&-
\sum_{i,j\in
I\ss}\left\langle\gc_{\tau_i}\nn,\tau_j\right\rangle^2=-\|B\cc\|^2\ngr.\end{eqnarray*}
\noindent {\it Proof of (ii)}. Since $\left\langle\nn,\tau_k\right\rangle=0$ for any
$k\in I\ss$ we
 get that
 $\left\langle\gc_{\tau_i}\nn,\tau_k\right\rangle=-\left\langle\nn,\gc_{\tau_i}\tau_k\right\rangle$ for every $i\in I\ss$.
  Therefore\begin{eqnarray*}\left\langle\gc_{\tau_i}\gc_{\tau_i}\nn,\tau_k\right\rangle+
   \left\langle\gc_{\tau_i}\nn,\gc_{\tau_i}\tau_k\right\rangle
 =-\left\langle\gc_{\tau_i}\nn,\gc_{\tau_i}\tau_k\right\rangle-\left\langle\nn,
 \gc_{\tau_i}\gc_{\tau_i}\tau_k\right\rangle.\end{eqnarray*}
Note that $\gc_{\tau_i}\nn\in\HS$ and that, by our  choice
of the moving frame, we have $\left(\gs_{\tau_i}\tau_k\right)(p)=0$.
Hence
\begin{eqnarray*}A_i:=\left\langle\gc_{\tau_i}\gc_{\tau_i}\nn,\tau_k\right\rangle
 &=&-\left\langle\nn,\gc_{\tau_i}\gc_{\tau_i}\tau_k\right\rangle
\\ &=&-\left\langle\nn,\gc_{\tau_i}\left(\gc_{\tau_k}\tau_i +
[\tau_i,\tau_k]\cc
 \right)\right\rangle\\&=&-\left\langle\nn,
 \gc_{\tau_i}\gc_{\tau_k}\tau_i\right\rangle-\left\langle\nn,\gc_{\tau_i}\left( \left\langle[\tau_i,\tau_k]\cc,\nn\right\rangle\nn
 \right)\right\rangle\quad\mbox{(by Lemma \ref{sple})}
 \\&=&-\left\langle\nn,\gc_{\tau_i}\gc_{\tau_k}\tau_i\right\rangle-\tau_i\left(\left\langle[\tau_i,\tau_k]\cc,\nn\right\rangle\right)\qquad\forall\,\,i,k\in I\ss.\end{eqnarray*}
Now since
$\left\langle[\tau_i,\tau_k]\cc,\nu\rr\right\rangle=\left\langle[\tau_i,\tau_k],\nu\rr\right\rangle=0$,
we get that
\begin{eqnarray*}\left\langle[\tau_i,\tau_k],\nn\right\rangle=
-\sum_{\alpha\in
I\vv}\varpi_\alpha\left\langle[\tau_i,\tau_k],\tau_\alpha\right\rangle
=-\sum_{\alpha\in I\vv}\varpi_\alpha C^\alpha_{ik}=\sum_{\alpha\in
I\vv}\varpi_\alpha \left\langle C^\alpha\cc\tau_i,\tau_k\right\rangle =
\left\langle C\ss(\varpi\cd)\tau_i,\tau_k\right\rangle\quad\forall\,\,i,k\in I\ss.\end{eqnarray*}Hence
$A_i=
-\left\langle\nn,\gc_{\tau_i}\gc_{\tau_k}\tau_i\right\rangle-\tau_i\left(\left\langle
C\ss(\varpi\cd)\tau_i,\tau_k\right\rangle\right).
$ Using  $\RC\cc=0$ 
(see Remark \ref{rh0} in Section \ref{prelcar}) yields
\begin{eqnarray*}\left\langle\gc_{\tau_i}\gc_{\tau_k}\tau_i,\nn\right\rangle=\left\langle\gc_{\tau_k}\gc_{\tau_i}\tau_i,\nn\right\rangle+
\left\langle\gc_{[\tau_i,\tau_k]\cc}\tau_i,\nn\right\rangle\qquad\forall\,\,i,k\in I\ss..
\end{eqnarray*}Therefore
$$\sum_{i\in I\ss}A_i=\underbrace{-\left\langle\nn,\gc_{\tau_k}\left(\sum_{i\in I\ss}\gc_{\tau_i}\tau_i\right)\right\rangle}_{=:A}
 -\sum_{i\in I\ss}\left(
\left\langle\gc_{[\tau_i,\tau_k]\cc}\tau_i,\nn\right\rangle+\tau_i\left(\left\langle
C\ss(\varpi\cd)\tau_i,\tau_k\right\rangle\right)\right).$$ We claim that $A=0$ at
$p$. Indeed, by hypothesis, $\MS=\left\langle\sum_{i\in
I\ss}\gc_{\tau_i}\tau_i,\nn\right\rangle$ is  constant. Hence we get that
$\left\langle\sum_{i\in
I\ss}\gc_{\tau_i}\tau_i,\gc_{\tau_k}\nn\right\rangle=0$ at $p$. Furthermore, since  at  $p\in
S$  one has
$[\tau_i,\tau_k]\cc=\left\langle C\ss(\varpi\cd)\tau_i,\tau_k\right\rangle\nn$ $\forall\,\,i,k\in I\ss$, it follows  that\begin{eqnarray*}\sum_{i\in I\ss}A_i&=&\sum_{i\in
I\ss}\left(\left\langle
C\ss(\varpi\cd)\tau_i,\tau_k\right\rangle\left\langle\gc_{\nn}\nn,\tau_i\right\rangle-\tau_i\left(\left\langle
C\ss(\varpi\cd)\tau_i,\tau_k\right\rangle\right)\right)\\&=&-\left(\left\langle\gc_{\nn}\nn,C\ss(\varpi\cd)\tau_k\right\rangle+\sum_{i\in
I\ss}\tau_i\left(\left\langle
C\ss(\varpi\cd)\tau_i,\tau_k\right\rangle\right)\right).\end{eqnarray*}Finally, (ii) will
follow from the next calculation:\begin{eqnarray*}\tau_i\left(\left\langle
C\ss(\varpi\cd)\tau_i,\tau_k\right\rangle\right)&=&\sum_{\alpha\in
I\vv}\left(\tau_i(\varpi_\alpha)\left\langle
C^\alpha\cc\tau_i,\tau_k\right\rangle +\varpi_\alpha\left(\left\langle
C^\alpha\cc\gc_{\tau_i}\tau_i,\tau_k\right\rangle+\left\langle
C^\alpha\cc\tau_i,\gc_{\tau_i}\tau_k\right\rangle\right)\right)\\&=&\sum_{\alpha\in
I\vv}\left(\tau_i(\varpi_\alpha)\left\langle
C^\alpha\cc\tau_i,\tau_k\right\rangle +\varpi_\alpha\left(-\left\langle
\gc_{\tau_i}\tau_i,\nn\right\rangle\left\langle
C^\alpha\cc\tau_k,\nn\right\rangle+\left\langle
C^\alpha\cc\tau_i,\nn\right\rangle\left\langle
\gc_{\tau_i}\tau_k,\nn\right\rangle\right)\right)\\&=&\sum_{\alpha\in
I\vv}\tau_i(\varpi_\alpha)\left\langle C^\alpha\cc\tau_i,\tau_k\right\rangle
+\MS\left\langle C\cc(\varpi\cd)\nn,\tau_k\right\rangle-B\cc(C\cc(\varpi\cd),\tau_k).
\end{eqnarray*}
\end{proof}
Using (i) of Lemma \ref{Doca}, yields the following \textquotedblleft folklore\textquotedblright\, result:
\begin{Prop}Let $S\subset\GG$ be a  $\cont^2$ hypersurface of constant horizontal mean
curvature $\MS$. Then, at each non-characteristic point $p\in S\setminus C_S$, we have: \begin{itemize}
 \item [{\rm(i)}] $ \left\langle\overrightarrow{\Delta\ss}\nn,\nn\right\rangle=-\|B\cc\|^2\ngr$;\item [{\rm(ii)}]
$  \overrightarrow{\Delta\ss} x\cc =\MS\nn$.
\end{itemize}

\end{Prop}

 Below we shall compute the
$\HS$-laplacian $\lll$ of the function $\f:=\left\langle V\cc, \nn\right\rangle$, where $V\cc\in\XH$ is a constant horizontal left invariant vector field.
\begin{lemma}\label{suppfunct}Let $S\subset\GG$ be  $\cont^2$ hypersurface of constant horizontal mean
curvature $\MS$. Then\begin{eqnarray*}-\lll\f=\f\|B\cc\|^2\ngr+\left\langle\gc_{\nn}\nn,C\ss(\varpi\cd) V\ss\right\rangle + \sum_{\alpha\in
I\vv}\left\langle C^\alpha\cc\qq\varpi_\alpha,V\ss\right\rangle+\MS\left\langle
C\cc(\varpi\cd)\nn,V\ss\right\rangle
\end{eqnarray*}at each non-characteristic point $p\in S\setminus C_S$.
\end{lemma}

\begin{proof}Fix a point $p\in S\setminus C_S$
and choose a moving frame centered at $p$; see Lemma \ref{sple}. We have
\begin{eqnarray*} \Delta\ss \f &=&\sum_{i\in
I\ss}\tau_i\tau_i(\left\langle V\cc,\nn\right\rangle) =\sum_{i\in
I\ss}\tau_i\left(
\left\langle V\cc,\gc_{\tau_i}\nn\right\rangle\right) =\sum_{i\in
I\ss}\left( \left\langle
V\cc,\gc_{\tau_i}\gc_{\tau_i}\nn\right\rangle\right)\\&=&-\Bigg(\f\|B\cc\|^2\ngr+\left\langle\gc_{\nn}\nn,C\ss(\varpi\cd) V\ss\right\rangle + \sum_{\alpha\in
I\vv}\left\langle C^\alpha\cc\qq\varpi_\alpha,V\ss\right\rangle\\&&+\MS\left\langle
C\cc(\varpi\cd)\nn,V\ss\right\rangle-B\cc(C\cc(\varpi\cd)\nn,V\ss)\Bigg),
\end{eqnarray*}
where we have used (i) and (ii) of Lemma \ref{Doca}. The thesis follows since$$B\cc(C\cc(\varpi\cd)\nn,V\ss)=-\left\langle
C\cc(\varpi\cd)\nn,\grad\ss\f\right\rangle.$$
\end{proof}

 A simple consequence of this lemma, at least from a \textquotedblleft formal\textquotedblright\, point of view, is that, in general, the function $\f$ cannot be an eigenfunction of a linear eigenvalue problem  
$\lll\varphi+\lambda{\mathcal B}\varphi=0$, where ${\mathcal B}$ is a given smooth function on $S\setminus C_S$.
This is a big difference compared with  the Euclidean case where, for any constant vector field $V\in\Rn$, the function $f=\left\langle V, \nu\right\rangle$ is \underline{always} a solution to the linear equation $\Delta\ts\varphi+ \|B\rr\|^2\ngr\varphi=0$. Here $\Delta\ts$ is the Laplace-Beltrami operator on $S$ and $B\rr$ is the 2nd fundamental form of $S$. This equation says that  $V$ is a \it Killing field \rm of any constant mean curvature hypersurface $S\subset\Rn$; see \cite{18}. Nevertheless, we have the following:

\begin{lemma}\label{vnct}Let $S\subset\GG$ be a $\cont^2$ hypersurface of constant horizontal mean
curvature. Then$$-\lll\varpi_\alpha=\varpi_\alpha \mathcal{B}\ts\qquad\forall\,\,\alpha\in I\vv $$at each non-characteristic point $p\in S\setminus C_S$.
\end{lemma}

\begin{proof}
For the sake of simplicity, we shall assume that $f$ is a NDF; \rm see Definition \ref{NDF}. Let $\underline{\tau}$ be an adapted moving frame along $S$. We have $\grad\cc f=\tau_1$ (and hence $\tau_1(f)=1$) and
$\tau_\alpha f=\varpi_\alpha$ for every $\alpha\in I\vv$. We stress that $\frac{\partial\varpi_\alpha}{\partial\tau_1}=X_\alpha\left(\frac{\partial f}{\partial \tau_1 }\right)=X_\alpha(1)=0$. Using Lemma \ref{ljjjkl} yields
\begin{eqnarray*} \Delta\ss \varpi_\alpha &=& \Delta\cc\varpi_\alpha -\left\langle{\rm
Hess}\cc(\varpi_\alpha)\tau_1,\tau_1\right\rangle\\  &=& \Delta\cc\left(\tau_\alpha f \right) -\left\langle{\rm
Hess}\cc\left(\tau_\alpha f \right)\tau_1,\tau_1\right\rangle\\  &=& \tau_\alpha\left( \Delta\cc(f)\right)  -\left\langle\nabla_{\tau_\alpha}\left({\rm
Hess}\cc(f)\right)\tau_1,\tau_1\right\rangle\\  &=& \varpi_\alpha\tau_1\left( \Delta\cc(f)\right)  -\left\langle\nabla_{\tau_\alpha}\left({\rm
Hess}\cc(f)\right)\tau_1,\tau_1\right\rangle \\&=& -\varpi_\alpha\tau_1\left(\MS\right)  -\left\langle\nabla_{\tau_\alpha}\left({\mathcal J}\cc\tau_1\right)\tau_1,\tau_1\right\rangle.
\end{eqnarray*}Since $\left\langle\left({\mathcal J}\cc\tau_1\right)\tau_1,\tau_1\right\rangle=0$, we get that
$\left\langle\nabla_{\tau_\alpha}\left(\left({\mathcal J}\cc\tau_1\right)\tau_1\right),\tau_1\right\rangle=-\left\langle\left({\mathcal J}\cc\tau_1\right)\tau_1,\nabla_{\tau_\alpha}\tau_1\right\rangle$ and hence
$$\left\langle \left({\mathcal J}\cc\tau_1\right)\nabla_{\tau_\alpha}\tau_1,\tau_1\right\rangle+\left\langle\nabla_{\tau_\alpha}\left({\mathcal J}\cc\tau_1\right)\tau_1,\tau_1\right\rangle=-\left\langle\left({\mathcal J}\cc\tau_1\right)\tau_1,\nabla_{\tau_\alpha}\tau_1\right\rangle\qquad\forall\,\,\alpha\in I\vv.$$But since $\left\langle \left({\mathcal J}\cc\tau_1\right)\nabla_{\tau_\alpha}\tau_1,\tau_1\right\rangle=0$, we obtain
 $$ \left\langle\nabla_{\tau_\alpha}\left({\mathcal J}\cc\tau_1\right)\tau_1,\tau_1\right\rangle=-\left\langle\left({\mathcal J}\cc\tau_1\right)\tau_1,\nabla_{\tau_\alpha}\tau_1\right\rangle=-\left\langle\gc_{\tau_1}\tau_1,\grad\cc\varpi_\alpha\right\rangle.$$
By using (i) of Lemma \ref{xzc}, it follows that $\gc_{\tau_1}\tau_1=-C\cc(\varpi\cd)\tau_1$ and so, by adding the quantity
$\left\langle C\cc(\varpi\cd)\tau_1, \qq \varpi_\alpha\right\rangle$, we finally get the identity
$ \lll\varpi_\alpha=-\varpi_\alpha\tau_1\left(\MS\right)$.
The quantity $\tau_1\left(\MS\right)$ can  be obtained   by repeating the calculations made in the proof of the 2nd variation formula.  We have\begin{eqnarray*}&&-\tau_1\left(\MS\right)\\&&= \div\ss\left(C\cc(\varpi\cd)\tau_1\right)
 - \left|C\cc(\varpi\cd)\tau_1\right|^2+\|A\cc\|\ngr^2-\|S\cc\|\ngr^2  + \sum_{j\in I\ss \alpha \in I\vv}\left\langle C\cc^\alpha\tau_1,\tau_j\right\rangle\left( \left\langle C(\varpi)\tau_\alpha,\tau_j\right\rangle- {\tau_j( \varpi_\alpha)}\right)\\&&= -\|A\cc\|\ngr^2-\|S\cc\|\ngr^2   -\sum_{\alpha\in I\vv} \left\langle
 \left(2 \qq(\varpi_\alpha)-C(\varpi)\TB_\alpha\right),C^\alpha
\tau_1\right\rangle \\&&=-\mathcal{B}\ts.
\end{eqnarray*}
\end{proof}

In Section \ref{exf}, just as an exercise, we will reprove this identity  for
the class of non-vertical hyperplanes $$\mathcal
I_{\alpha'}:=\left\{p\equiv\exp\left(\sum_{j=1}^n x_j\right)\in\GG:
x_{\alpha'}=0\right\},$$ where $\alpha'\in I\vv$; see Definition \ref{hyppl}. For the sake of
simplicity, this will be done only for $2$-step Carnot groups.
We recall that these hyperplanes are very different from the
vertical ones and, for instance, they turn out to be
characteristic at the identity $0\in\GG$. We now state an immediate consequence of the previous lemma. To this aim, let $V\in\XG$ be a constant left invariant vector field.

\begin{corollario}\label{vnct2}Let $S\subset\GG$ be a $\cont^2$ hypersurface of constant horizontal mean
curvature. Then the function $f\vv:=\left\langle V,
\varpi\right\rangle$ satisfies the equation $-\lll f\vv=f\vv
\mathcal{B}\ts$ at each non-characteristic point of $S$.\end{corollario}

\section{Stability of $\HH$-minimal hypersurfaces}\label{SR1}
 
\begin{Defi}[Stability]\label{strdef}Let $\GG$ be a $k$-step Carnot group and let $S\subset\GG$ be a $\HH$-minimal
 hypersurface of class  $\cont^3$.
\begin{itemize}
\item[$(S_1)$]If $C_S=\emptyset$, we say that $S$ is  {\rm stable} if $II_{S}(\per)\geq 0$ for every
  $\cont^3$-smooth compactly supported  variation $\vartheta_t:]-\epsilon, \epsilon[\times S\longrightarrow\GG$.
 \item[$(S_2)$]
If $C_S\neq\emptyset$, we further assume that $\frac{1}{|\PH\nu|^2}\in
L^2(S, \sigma\rr^{n-1})$. Then, we say that $S$ is  {\rm stable} if $II_{S}(\per)\geq 0$ for every
  $\cont^3$-smooth compactly supported variation $\vartheta_t:]-\epsilon, \epsilon[\times S\longrightarrow\GG$  for which
  there exists
$h\in L^1\left( \UU; \sigma\rr^{n-1}\right)$ such that $\frac{1}{|\P\ct\nu^t|}\leq h$ for every $t\in]-\epsilon, \epsilon[$. 
\end{itemize}
\end{Defi} 
\begin{oss}
We shall sometimes  say that $S$ is \rm strictly stable \it when
the stability inequality is strict. If $C_S\neq\emptyset$, but we  use only compactly supported  variations on $S^\ast:=S\setminus C_S$, then  $(S_1)$  applies to any non-characteristic domain $\Om\Subset S^\ast$.
\end{oss}

\begin{lemma}Let $S\subset\GG$ be as in Definition \ref{strdef} and let us consider the following \rm linear eigenvalue problem, \it i.e. \begin{eqnarray*}\left\{\begin{array}{ll} \lll\varphi+\lambda\,\mathcal{B}\ts \varphi =0\quad \mbox{on}\,\,S\\\qquad\qquad\,\,\,\,\quad \varphi=0\quad \mbox{on}\,\,\partial S.\end{array}\right.\end{eqnarray*}Under the previous assumptions, a  sufficient condition for  stability  of $S$ is
 that the first (non-trivial) eigenvalue $\lambda_1$ of this problem is greater than or equal to $1$; see Notation \ref{415}.
\end{lemma}\begin{proof}This is an immediate consequence of the horizontal Green formula \eqref{eq}; see Corollary \ref{GDc}.

\end{proof}

\begin{lemma}\label{zxc}Let $S\subset\GG$ be a hypersurface of class $\cont^2$. Let $\Om\subset S^\ast=S\setminus C_S$ be a  bounded non-characteristic domain and let $q\in\cont(\Om)$. If there exists a smooth  function $\psi>0$ on $\Om$ satisfying the equation $\lll \psi= q \psi$, then \begin{equation}\label{ininin}\int_{\Om}\left(|\qq \varphi|^2+ q \varphi^2\right)\,\per\geq 0\end{equation}
for all smooth function $\varphi$ compactly supported on $\Om$.
\end{lemma}

This lemma generalizes a well-known result in the Riemannian setting; see \cite{FCS}.

\begin{proof}[Proof of Lemma \ref{zxc}]If $\psi>0$ satisfies $\lll \psi= q \psi$ on $\Om$, let us define a new function $\phi:=\log \psi$.
By an elementary calculation we see that $\lll\phi = q- |\qq\phi|^2$. More precisely, we have
\begin{eqnarray*}\lll\phi&=&\div\ss\left(\qq\phi\right)+\left\langle C\cc(\varpi\cd)\nn, \qq\phi\right\rangle\\
&=&\div\ss\left(\dfrac{\qq\psi}{\psi}\right)+ \left\langle C\cc(\varpi\cd)\nn,\dfrac{\qq\psi}{\psi}\right\rangle
\\&=&\left(\dfrac{\Delta\ss\psi}{\psi}\left\langle C\cc(\varpi\cd)\nn,\dfrac{\qq\psi}{\psi}\right\rangle\right)-
\dfrac{|\qq\psi|^2}{\psi^2}\\&=& \dfrac{\lll\psi}{\psi}-{|\qq\phi|^2}
\\&=& q-{|\qq\phi|^2}.
\end{eqnarray*}
So let $\varphi$ be a smooth function with compact support on $\Om$. Multiplying by $-\varphi^2$ both sides of this equation and integrating by parts, yields
\begin{eqnarray}\label{jkhgd}-\int_{\Om}\varphi^2\left(q-|\qq\phi|^2\right)\,\per=-\int_{\Om}\varphi^2\lll\phi\,\per=
\int_{\Om}2\varphi\left\langle\qq\varphi,\qq\phi\right\rangle\,\per,
\end{eqnarray}where  we have used Corollary \ref{GDc}.
Note that\begin{equation}\label{inni}
2|\varphi\left\langle\qq\varphi,\qq\phi\right\rangle|\leq
2|\varphi||\qq\varphi||\qq\phi|\leq
|\varphi|^2|\qq\phi|^2+|\qq\varphi|^2.
\end{equation}Hence, \eqref{ininin} follows  inserting \eqref{inni}  into \eqref{jkhgd} and 
canceling the terms $\int_{\Om}\varphi^2|\qq\phi|^2\,\per$.
\end{proof}

As a consequence of Lemma \ref{vnct} and  Lemma \ref{zxc}, we  infer an interesting  condition for stability.

\begin{teo} \label{MRS}
Let $S\subset\GG$ be a
$\HH$-minimal hypersurface of class $\cont^3$. If there exists $\alpha\in I\vv$ such that either $\varpi_\alpha>0$ or $\varpi_\alpha<0$ on $S$, then each non-characteristic domain $\Om\subset 
S^\ast$ turns out to be stable. \end{teo}
\begin{proof}By applying Lemma \ref{zxc} to the function $\varpi_\alpha$ we immediately get the stability inequality $$II_{S}(W,\per)\geq 0$$  for every
non-zero compactly supported  variation $\vartheta_t$  of
$S$.
 \end{proof}
We have the following reformulations of Theorem \ref{MRS}:
 \begin{corollario}Let $S\subset\GG$ be a
 $\HH$-minimal hypersurface of class $\cont^3$.  Let $V\in\XG$ be a
constant left invariant vector field and set
$f\vv=\left\langle V, \varpi\right\rangle$.  If either $f\vv>0$  or $f\vv<0$, then  each non-characteristic domain $\Om\subset S^\ast$ is stable.
\end{corollario}
\begin{corollario}Let  $S\subset \GG$  be a complete $\HH$-minimal hypersurface of class $\cont^3$. If $S$ is a graph with respect to a given vertical direction, then each non-characteristic domain $\Om\subset S^\ast$ is stable.
\end{corollario}

Below  we shall study some (more or less simple) examples in order to illustrate some of our results.

\subsection{Examples}\label{exf}

Our first example, which is that of \it vertical hyperplanes, \rm is the simplest one and, to the best of our knowledge, the only known in literature outside the Heisenberg group setting. Roughly speaking, vertical hyperplanes are level-sets of linear homogeneous polynomial having (homogeneous) degree 1, which are ideals of the Lie algebra $\gg$.

 We claim that they are (strictly) stable hypersurfaces. This immediately follows from the fact that ${\mathcal B}\ts=0$. Hence, for any regular bounded domain $\UU$ contained on a vertical hyperplane $\mathcal I$, we have
$II_\UU(W,\per)=\int_\UU|\qq\cn|^2\,\per\geq 0$, with equality if, and only if,  $\cn=0$.
\begin{corollario}\label{c0}  Let $\GG$ be a $k$-step Carnot group. A vertical hyperplane is a $\cin$-smooth non-characteristic  strictly stable $\HH$-minimal hypersurface.
\end{corollario}

Now we analyze a completely different family of hyperplanes. From an intrinsic point of view, they are homogeneous ``cones'', which turn out to be characteristic at a single point. For the sake of simplicity, we just consider the  case of  2-step Carnot groups. We have  $\gg=\HH\oplus\VV\,\,(\dim\HH=\DH,\,\dim\VV=n-\DH)$. Let us assume that $$X_i(x):=\ee_i+\frac{1}{2}\sum_{\alpha\in I\vv}\left\langle C^\alpha\cc\ee_i, x\cc\right\rangle\ee_\alpha,\qquad X_\alpha=\ee_\alpha$$
for every $i\in I\cc=\{1,...,\DH\}$ and every $\alpha\in I\vv=\{\DH+1,...,n\}$, where $\ee_j=(0,...,\underbrace{1}_{j-th place},...0),\, j=1...,n$, is the $j$-th vector of the canonical basis of $\Rn\cong\gg$ and  $x\cc\equiv(x_1,...,x_\DH)$ is the horizontal position vector. As usual, we identify vector fields and differential operators.

Fix $\alpha'\in I\vv$ and  consider the \it non-vertical hyperplane \rm  $\mathcal I_{\alpha'}:=\left\{x=\esp\left(\sum_j x_j\right)\in\GG:   x_{\alpha'}=0\right\}.$
We have $\grad\cc x_{\alpha'}=-\frac{1}{2}C\cc^{\alpha'} x\cc$ and so $\nn=\frac{-C\cc^{\alpha'} x\cc}{|C\cc^{\alpha'} x\cc|}$. Moreover, $\varpi_\beta=0$ for all $\beta\neq\alpha'$ and $\varpi_{\alpha'} = \frac{2}{|C\cc^{\alpha'} x\cc|}$.
Since $$\div\cc \left(C\cc^{\alpha'}x\cc\right)=\sum_{j\in I\cc}\left\langle\nabla_{X_j}C\cc^{\alpha'}x\cc, X_j\right\rangle=\sum_{j\in I\cc}\left\langle C\cc^{\alpha'} X_j, X_j\right\rangle=0$$and $$\left\langle\grad\cc\left(\dfrac{1}{|C\cc^{\alpha'}x\cc|}\right), C\cc^{\alpha'} x\cc\right\rangle=-\left\langle \dfrac{\grad\cc|C\cc^{\alpha'}x\cc|}{|C\cc^{\alpha'}x\cc|^2}, C\cc^{\alpha'} x\cc\right\rangle=\left\langle\dfrac{C\cc^{\alpha'}\nn}{|C\cc^{\alpha'}x\cc|}, \nn\right\rangle=0,$$it follows that $\MH=-\div\cc\nn=0$, i.e.  $\mathcal I_{\alpha'}$ is {\it $\HH$-minimal}. The above calculation also shows that $ \grad\cc\left({|C\cc^{\alpha'}x\cc|}\right) =C\cc^{\alpha'}\nn$. Furthermore, we easily get that
\begin{eqnarray*}-\mathcal{J}\cc\nn= \dfrac{C\cc^{\alpha'}+\nn\otimes C\cc^{\alpha'}\nn}{|C\cc^{\alpha'} x\cc|},
\end{eqnarray*}which, in turn,  implies $$B\cc(\tau_i, \tau_j)=\left\langle \dfrac{C\ss^{\alpha'}}{|C\cc^{\alpha'} x\cc|}\tau_i, \tau_j\right\rangle=A\cc(\tau_i, \tau_j)\qquad \forall\,\,i, j\in I\ss.$$Therefore  $S\cc=\bf 0\cc$ (i.e. the 0-matrix on $\HH$) and
  $\|B\cc\|\ngr^2=\|A\cc\|\ngr^2=\dfrac{\varpi_{\alpha'}^2\|C\ss^{\alpha'}\|\ngr^2}{4}$.
 It remains to compute the quantity $\Upsilon:= -\sum_{\alpha\in I\vv} \left\langle
 \left(2 \qq(\varpi_\alpha)-C(\varpi)\TB_\alpha\right),C^\alpha
\tau_1\right\rangle$; see formula \eqref{tranx}. Since we are in a $2$-step group, we have
\begin{eqnarray*}  \Upsilon= -\sum_{\alpha\in I\vv} \left\langle
 \left(2 \qq(\varpi_\alpha)+\varpi_\alpha C(\varpi)\tau_1\right),C^\alpha
\tau_1\right\rangle. \end{eqnarray*}From the previous calculations, it follows that $\Upsilon=0$ and so $\mathcal B\ts=\|A\cc\|\ngr^2$. In other words, we have\begin{eqnarray*}\label{svex}  II_\UU(W,\per)= \int_{\UU}\left(
 |\qq\cn|^2  -\cn^2 \|A\cc\|\ngr^2 \right)\,\per=\int_{\UU}\left(
 |\qq\cn|^2  -\cn^2 \frac{\varpi_{\alpha'}^2\|C\ss^{\alpha'}\|\ngr^2}{4} \right)\,\per\end{eqnarray*}for any  non-characteristic  bounded domain $\UU\subset \mathcal I_{\alpha'}(\equiv S)$, where $\per\res \mathcal I_{\alpha'}=\frac{|C\cc^{\alpha'} x\cc|}{2}\,d{\mathcal L}\eu^{n-1} \res \mathcal I_{\alpha'}$ and  $$d\mathcal L\eu^{n-1} =dx_1\wedge dx_2\wedge...\wedge...\wedge\widehat{dx_{\alpha'}}\wedge...\wedge dx_n.$$
It goes without saying that the previous formula holds true near the characteristic set only under the assumptions made in Corollary \ref{corvar2stat}. In particular, we have to check that $\int_\UU\frac{1}{|\PH\nu|^4}\,\sigma\rr^{n-1}<+\infty$, which is clearly equivalent to the next condition: \begin{equation}\label{bhhb}\int_\UU\frac{1}{|C\cc^{\alpha'} x\cc|^4}\,d\mathcal L\eu^{n-1}\res \mathcal I_{\alpha'}<+\infty.\end{equation}

Two remarks are in order:
\begin{itemize}
 \item a necessary condition for the validity of \eqref{bhhb} is that the dimension of $\HH$ is $\geq 5$, i.e. \begin{equation}\label{bhhb2} \DH=\dim\HH\geq 5;
\end{equation}
\item in the Heisenberg group $\mathbb H^n$, the previous analysis reduces to the case of the horizontal hyperplane $\left\lbrace p=\exp(z,t)\in\mathbb H^n: t=0 \right\rbrace$ and, in this case, \eqref{bhhb2} is also  sufficient  for \eqref{bhhb} to hold.
\end{itemize}

Now let us compute\begin{eqnarray*}\Delta\ss\left(\dfrac{1}{|C\cc^{\alpha'}x\cc|}\right)&=&-\div\ss\left(\dfrac{
C\cc^{\alpha'}\nn}{|C\cc^{\alpha'}x\cc|^2}\right)\\&=&\dfrac{2|C\cc^{\alpha'}\nn|^2}{|C\cc^{\alpha'}x\cc|^3}-\dfrac{\div\ss(C\cc^{\alpha'}\nn)}{|C\cc^{\alpha'}x\cc|^2}\\&=&\dfrac{2|C\cc^{\alpha'}\nn|^2}{|C\cc^{\alpha'}x\cc|^3}+\dfrac{\sum_{j, k
\in I\ss}\left\langle\gc_{\tau_j}\tau_1,\tau_k\right\rangle\left\langle C\cc^{\alpha'}{\tau_j},\tau_k\right\rangle}{|C\cc^{\alpha'}x\cc|^2}\\&=&\dfrac{2|C\cc^{\alpha'}\nn|^2-\|C\ss^{\alpha'}\|^2\ngr}{|C\cc^{\alpha'}x\cc|^3}.\end{eqnarray*}From this computation and the very definition of $\lll$, it follows that$$\lll\left(\dfrac{1}{|C\cc^{\alpha'}x\cc|}\right)=\Delta\ss\left(\dfrac{1}{|C\cc^{\alpha'}x\cc|}\right) +  \left\langle C\cc(\varpi\cd)\nn,
\qq\left(\dfrac{1}{|C\cc^{\alpha'}x\cc|}\right)\right\rangle =-\dfrac{\|C\ss^{\alpha'}\|^2\ngr}{|C\cc^{\alpha'}x\cc|^3},$$which is equivalent to the equation $\lll \varpi_{\alpha'}=-\varpi_{\alpha'} \|A\cc\|\ngr^2$, as predicated by Lemma \ref{vnct}.

The previous discussion is summarized in the following:
\begin{corollario}\label{c1} Let $\GG$ be a $2$-step Carnot group and let $\mathcal I_{\alpha'}$ be a horizontal hyperplane passing through $0\in\GG$. Then  $\mathcal I_{\alpha'}$ is a $\cin$-smooth  $\HH$-minimal hypersurface. The only characteristic point of $\mathcal I_{\alpha'}$ is the identity $0\in\GG$. Furthermore,  any bounded  domain $\UU\Subset\mathcal I_{\alpha'}\setminus \{0\}$ turns out to be strictly  stable.
\end{corollario}

\subsection{An example in the Heisenberg group $\mathbb H^n$}\label{exf2}
For the notation used in this section we refer the reader to Example \ref{hng} and Example \ref{hng2}.
We recall that any point $p\in\mathbb H^n$ is identified with $(z,t)\in\R^{2n+1}$, where $z=(x_1,y_1,x_2,y_2,...,x_n,y_n)$.
We use the following further notation: $$\overline{v}^{1,0}:=(v_1,0,v_2,0,...,v_n,0)\in{\R^{2n}},\qquad\overline{v}^{0,1}:=(0,v_1,0,v_2,0,...,0, v_n)\in{\R^{2n}}\qquad\forall\,\,v=(v_1,v_2,...,v_n)\in\Rn.$$Using this notation yields $z=\overline{x}^{1,0}+\overline{y}^{0,1}\in\R^{2n}$, where $x=(x_1,x_2,...,x_n)\in\Rn$ and $y=(y_1,y_2,...,y_n)\in\Rn$.\\
In the sequel, we shall study the following \it hyperbolic paraboloid: \rm \begin{equation}
 \label{sex}
S:=\left\lbrace p\equiv(z,t)\in\mathbb H^n: t=\dfrac{\left\|\overline{x}^{1,0}\right\|_{\Rn}^2-\left\|\overline{y}^{1,0}\right\|_{\Rn}^2}{4}\right\rbrace,
\end{equation}
  First, note that $\grad\cc t=\frac{z^\circ}{2}$, where  $z^\circ:=-{C\cc^{2n+1} z}$. Furthermore, a simple calculation shows that
$\grad\cc\left( \frac{\left\|\overline{x}^{1,0}\right\|_{\Rn}^2-\left\|\overline{y}^{1,0}\right\|_{\Rn}^2}{4}\right)=\frac{1}{2}\left(\overline{x}^{1,0}-
\overline{y}^{0,1} \right)$ and hence $\nn= \frac{-\overline{v}^{1,0}+\overline{v}^{0,1}}{|-\overline{v}^{1,0}+\overline{v}^{0,1}|}$, where we have set $v=x+y\in\Rn$. Therefore $$\nn=\dfrac{\sqrt{2}}{2}\left(\dfrac{-\overline{v}^{1,0}+\overline{v}^{0,1}}{\sqrt{\rho^2 + 2\left\langle x, y\right\rangle_{\Rn}}}\right),\qquad\nn^\circ=-\dfrac{\sqrt{2}}{2}\left(\dfrac{\overline{v}^{1,0}+\overline{v}^{0,1}}{\sqrt{\rho^2 + 2\left\langle x, y\right\rangle_{\Rn}}}\right),$$ where $\sqrt{\rho^2 + 2\left\langle x, y\right\rangle_{\Rn}}=\|x+y\|_{\Rn}$ and $\rho:=\sqrt{\|x\|_{\Rn}^2+\|y\|_{\Rn}^2}$. Clearly, the characteristic set $C_S$ of $S$ is the set of all points $p\equiv(z,t)\in S$ such that $\overline{x+y}^{1,0}=\overline{x+y}^{0,1}=0\in\R^{2n}$.  Hence  $p\equiv(z,t)\in C_S$ if, and only if $x_i=-y_i$ for every $i=1,...,n$. Since $X_i\left( \frac{1}{\|x+y\|_{\Rn}}\right)=Y_i\left( \frac{1}{\|x+y\|_{\Rn}}\right)$, we easily get  that $\div\cc\nn=0$, i.e. $S$ is \it  $\HH$-minimal. \rm

We have $\varpi=\frac{\sqrt{2}}{\|x+y\|_{\Rn}}$ and $\frac{\partial\varpi}{\partial\nn^\circ}=\frac{2}{\|x+y\|^2_{\Rn}}.$
In order to calculate the horizontal 2nd fundamental form $B\cc$ (and some of its invariants) we need  the horizontal Jacobian matrix $\mathcal J\cc \nn=:\left[ a_{i j} \right]_{i, j\in  I\cc} $ of the  $\HH$-normal $\nn$. For the sake of simplicity, we treat the case $n=2$, which corresponds to the 2nd Heisenberg group. The general case is completely analogous. We have
\begin{itemize}
 \item $a:=a_{11}=a_{12}=-\frac{\sqrt{2}}{2}\left(\frac{\|x+y\|^2_{\R^2}-(x_1+y_1)^2}{\|x+y\|^3_{\R^2}} \right) $,\quad $b:=a_{13}=a_{14}=-\frac{\sqrt{2}}{2}\left(\frac{-(x_1+y_1)(x_2+y_2)}{\|x+y\|^3_{\R^2}} \right) $;
\item $a_{2j}=-a_{1j}$ for every $j=1,...,4$;
\item $a_{31}=a_{32}=-\frac{\sqrt{2}}{2}\left(\frac{-(x_1+y_1)(x_2+y_2)}{\|x+y\|^3_{\R^2}} \right)$, \quad $c:=a_{33}=a_{34}=-\frac{\sqrt{2}}{2}\left(\frac{\|x+y\|^2_{\R^2}-(x_2+y_2)^2}{\|x+y\|^3_{\R^2}} \right)$;\item $a_{4j}=-a_{3j}$ for every $j=1,...,4$.
\end{itemize}
Equivalently,  $\mathcal J\cc \nn=\left[
\begin{array}{cccc}
  a & a &  b &  b \\
 -a & -a & -b & -b \\
   b & b & c & c \\
  -b & -b & -c & -c
\end{array}%
\right]$. It follows that $\nn\in{\mathrm Ker}\mathcal J\cc \nn$ and hence  $B\cc=-\mathcal J\cc \nn$. By definition, we have
$S\cc= -\left(\frac{\mathcal J\cc \nn +\left(\mathcal J\cc \nn\right)^{\rm Tr}}{2}\right)=-\left[
\begin{array}{cccc}
  a & 0 &  b &  0 \\
 0 & -a & 0 & -b \\
   b & 0 & c & 0 \\
  0 & -b & 0 & -c
\end{array}%
\right] $. So if $n=2$, we have$$\|B\cc\|\ngr^2=4\left( a^2+2b^2+c^2\right)=\frac{2}{\|x+y\|^2_{\Rn}},\qquad \|S\cc\|\ngr^2=2\left( a^2+2b^2+c^2\right)=\frac{1}{\|x+y\|^2_{\Rn}}=\|A\cc\|\ngr^2.$$In the general case, an analogous calculation gives
$\|B\cc\|\ngr^2=\frac{2(n-1)}{\|x+y\|^2_{\Rn}}$ and $\|S\cc\|\ngr^2=\|A\cc\|\ngr^2=\frac{n-1}{\|x+y\|^2_{\Rn}}$.
Therefore, using \eqref{sadf} yields \begin{eqnarray*}
\mathcal B\ts&=&\|S\cc\|\ngr^2-\left(
2\frac{\partial\varpi}{\partial\nn^\circ} -\frac{n+1}{2}\varpi^2
\right)\\&=&\frac{n-1}{\|x+y\|^2_{\Rn}}-\left(
 \frac{4}{\|x+y\|^2_{\Rn}} -\frac{n+1}{\|x+y\|^2_{\Rn}}
\right)\\&=&\frac{2(n-2)}{\|x+y\|^2_{\Rn}}.
\end{eqnarray*}

\begin{oss}[$ B\ts$ can be negative!]\label{urme} If $n=1$, then $ B\ts<0$ for any non-characteristic  domain  $\UU\subset S$. 
\end{oss}

The previous calculation implies that
\begin{eqnarray}\label{svex}  II_\UU(W,\perh)= \int_{\UU}\left(
 |\qq\cn|^2  -\cn^2\frac{2(n-2)}{\|x+y\|^2_{\Rn}} \right)\,\perh,\end{eqnarray}for any non-characteristic bounded domain  $\UU\subset S$, where $\perh=\frac{\|x+y\|_{\Rn}}{\sqrt{2}}\,dz$ and we have set $$dz=dx_1\wedge dy_1\wedge...\wedge dx_n\wedge dy_n.$$

\begin{oss}[Failure of $\int_\UU\frac{1}{|\PH\nu|^4}\,\sigma\rr^{n-1}<+\infty$ for characteristic domains]In order to apply the previous 2nd variation formula for a characteristic domain $\UU \subset S$, we need (at least) to check that \begin{equation}\label{bhhb22}\int_\UU\frac{1}{\|x+y\|_{\Rn}^4} \,d\mathcal L\eu^{2n}< +\infty.\end{equation}However, in general, this condition fails to hold if $C_\UU\neq \emptyset$.
\end{oss}

Lemma \ref{ljjjkl} says $\Delta\ss\varpi=\Delta\cc\varpi-\left\langle{\rm Hess}\cc\varpi\,\nn,\nn\right\rangle$. Since
$\grad\cc\varpi=-\sqrt{2}\left( \frac{\overline{(x+y)}^{1,0}+\overline{(x+y)}^{0,1}}{\|x+y\|_{\Rn}^3}\right)$, we easily get that $\Delta\cc\varpi=-\varpi\frac{(2n-3)}{\|x+y\|_{\Rn}^2}$. Furthermore
$$\left\langle{\rm Hess}\cc\varpi\,\nn,\nn\right\rangle=\left\langle\left[
\begin{array}{ccccc}
  1 & 1 &  0 &  0 & \ldots \\
 1 & 1 & 0 & 0 & \ldots \\
   0 & 0 & 1 & 1& \ldots  \\
  0 & 0 & 1 & 1& \ldots\\ \ldots& \ldots&\ldots&\ldots&\ldots
\end{array}%
\right]\,\nn,\nn\right\rangle=-\frac{\varpi}{\|x+y\|_{\Rn}^2}.$$All together,  we have shown that $$\lll\varpi=\Delta\ss\varpi-\varpi\frac{\partial\varpi}{\partial\nn^\circ}=-\varpi\frac{2(n-2)}{\|x+y\|^2_{\Rn}},$$
 which illustrates the content of Lemma \ref{vnct}.

\begin{corollario}\label{c3} Let $S=\Bigg\lbrace p\equiv(z,t)\in\mathbb H^n: t=\frac{\|\overline{x}^{1,0} \|_{\Rn}^2- \|\overline{y}^{1,0} \|_{\Rn}^2}{4}\Bigg\rbrace$, where  $z:=\overline{x}^{1,0}+\overline{y}^{0,1}\in\R^{2n}$.  Then $S$ turns out to be a $\cin$-smooth $\HH$-minimal hypersurface. Furthermore, one has $$C_S=\left\{p=\exp(z,t)\in S:  x_i=-y_i,\,\,i=1,...,n\right\}.$$ Finally, any  bounded  domain $\UU\Subset S\setminus C_S$ is  strictly  stable.
\end{corollario}

\rm

{\footnotesize \noindent Francescopaolo Montefalcone:
\\Dipartimento di Matematica \\Universit\`a degli Studi di Padova,\\
  Via Trieste, 63, 35121 Padova (Italy)\,
 \\ {\it E-mail address}:  {\textsf montefal@math.unipd.it}}

\end{document}